\RequirePackage{fix-cm}
\documentclass[leqno, twoside]{article}
\usepackage{geometry}
\usepackage{amsmath, amsthm, amssymb, dsfont, mathrsfs}
\usepackage{graphicx, caption, float, array, multirow}
\usepackage[hidelinks]{hyperref}

\numberwithin{equation}{subsection}

\usepackage{fourier, anyfontsize}
\usepackage[T1]{fontenc}

\usepackage{chngcntr, etoolbox}
\counterwithin*{equation}{section}
\counterwithin*{equation}{subsection}

\allowdisplaybreaks

\usepackage{fancyhdr}
\newtheoremstyle{theorem}
  {10pt}
  {10pt}
  {\sl}
  {\parindent}
  {\bf}
  {. }
  { }
  {}
\theoremstyle{theorem}
\newtheorem{theorem}{Theorem}
\newtheorem{corollary}{Corollary}

\begin{document}\large

\title{\textbf{Average-tempered stable subordinators \\ with applications}}
\author{Weixuan Xia\footnote{Correspondence address: Department of Finance, Boston University Questrom School of Business. Email: \underline{gabxia@bu.edu}.}}
\date{2020}
\maketitle
\thispagestyle{plain}

\begin{abstract}
  In this paper the running average of a subordinator with a tempered stable distribution is considered. We investigate a family of previously unexplored infinite-activity subordinators induced by the probability distribution of the running average process and determine their jump intensity measures. Special cases including gamma processes and inverse Gaussian processes are discussed. Then we derive easily implementable formulas for the distribution functions, cumulants, and moments, as well as provide explicit estimates for their asymptotic behaviors. Numerical experiments are conducted for illustrating the applicability and efficiency of the proposed formulas. Two important extensions of the running average process and its equi-distributed subordinator are examined with concrete applications to structural degradation modeling and financial derivatives pricing, where their advantages relative to several existing models are highlighted together with the mention of Euler discretization and compound Poisson approximation techniques. \vspace{0.1in} \vspace{0.1in}
\end{abstract}

\vspace{0.2in}

\newcommand{\dd}{{\rm d}}
\newcommand{\pd}{\partial}
\newcommand{\ii}{{\rm i}}
\newcommand{\E}{\mathbb{E}}
\newcommand{\PP}{\mathbb{P}}
\newcommand{\Var}{\mathrm{var}}
\newcommand{\Skew}{\mathrm{skew}}
\newcommand{\EKurt}{\mathrm{ekurt}}
\newcommand{\Cov}{\mathrm{cov}}
\newcommand{\EE}{\mathrm{E}}
\newcommand{\Gf}{\mathrm{\Gamma}}
\renewcommand{\Re}{\mathrm{Re}}
\renewcommand{\Im}{\mathrm{Im}}
\renewcommand{\Lambda}{\varLambda}
\renewcommand{\Pi}{\varPi}

\section{Introduction}\label{sec:1}

Let $(\Omega,\mathcal{F},\PP;\mathbb{F})$ be a continuous-time stochastic basis, where the filtration $\mathbb{F}\equiv\{\mathscr{F}_{t}\}_{t\geq0}$ satisfies the usual conditions and with respect to which all stochastic processes are by default considered in this paper. A tempered stable subordinator $X\equiv(X_{t})_{t\geq0}$ is defined as a L\'{e}vy process with a marginal one-sided tempered stable distribution. That said, it is immediately understood that $X$ starts from 0, $\PP$-a.s., has independent and stationary increments, and has $\PP$-\text{a.s.} c\`{a}dl\`{a}g sample paths.

The one-sided tempered stable distribution forms an infinitely divisible family with three parameters - $a>0$, $b>0$, and $c\in(0,1)$, denoted by $\mathrm{TS}(a,b;c)$. More specifically, the random variable $X_{t}$ admits the following Laplace transform,
\begin{equation}\label{1.1}
  \bar{f}_{X_{t}}(u)\equiv\bar{f}_{X_{t}}(u|a,b;c):=\E e^{-uX_{t}}=\exp(at\Gf(-c)((b+u)^{c}-b^{c})),\quad u\in\mathds{C}\setminus(-\infty,-b],\;t\geq0,
\end{equation}
which clearly satisfies $\bar{f}_{X_{t}}(u)=(\bar{f}_{X_{1}}(u))^{t}$, for every $t\geq0$. In particular, we shall write
\begin{equation*}
  X_{t}\overset{\rm d.}{=}\mathrm{TS}(at,b;c),\quad t>0.
\end{equation*}
The tempered stable subordinator is known to be a purely discontinuous nondecreasing L\'{e}vy process with no Brownian component. According to the L\'{e}vy-Khintchine representation, we have the Laplace exponent
\begin{equation*}
  \log\bar{f}_{X_{t}}(u)=t\bigg(-\alpha u+\frac{\beta^{2}u^{2}}{2}+\int_{\mathds{R}\setminus\{0\}}\big(e^{-ux}-1+ux\mathds{1}_{\{|x|<1\}}\big)\nu(\dd x)\bigg),\quad t\geq0,
\end{equation*}
in which the L\'{e}vy triplet $(\alpha,\beta,\nu)$\footnote{In particular, we refer to $\alpha$ as the drift component, $\beta$ the Brownian component, and $\nu$ the jump intensity measure or L\'{e}vy measure.} of $X$, for constant $\alpha\in\mathds{R}$ and $\beta\geq0$ and some Borel measure $\nu$ on $\mathds{R}\setminus\{0\}$ satisfying $\int_{\mathds{R}\setminus\{0\}}(1\wedge x^{2})\nu(\dd x)<\infty$, is specified by
\begin{align}\label{1.2}
  \alpha&=ab^{c-1}(\Gf(1-c)-\Gf(1-c,b)),\\
  \beta&=0, \nonumber\\
  \ell(x):=\frac{\nu(\dd x)}{\dd x}&=\frac{ae^{-bx}}{x^{c+1}}\mathds{1}_{(0,\infty)}(x),\quad x\in\mathds{R}, \nonumber
\end{align}
where $\Gf(\cdot)$ and $\Gf(\cdot,\cdot)$ denote the usual gamma function and the upper incomplete gamma function, respectively, and $\ell$ is termed the L\'{e}vy density associated with $\nu$. Clearly, the L\'{e}vy density is itself not integrable but has a finite Blumenthal-Getoor index\footnote{This index was initially introduced in [Blumenthal and Getoor, 1961] \cite{BG} as a measure of the path smoothness of a L\'{e}vy process. The larger the index, the more irregular the paths.} given by,
\begin{equation*}
  \mathbf{B}(X):=\inf\bigg\{p>0:\int_{|x|<1}|x|^{p}\ell(x)\dd x<\infty\bigg\}=c<1,
\end{equation*}
which indicates that $X$ is an infinite-activity process whose sample paths are of finite variation.

Many well-known L\'{e}vy processes are special cases of the tempered stable subordinator, including the gamma process and the inverse Gaussian process, which are obtained for $c\searrow0$ and $c=1/2$, respectively. Indeed, we have $\mathrm{TS}(at,b;c)\rightarrow\mathrm{G}(at,b)$ as $c\searrow0$ and $\mathrm{TS}(at,b;1/2)=\mathrm{IG}(at,b)$, and write $(X_{t}|c\searrow0)=(G_{t})=G$ and $(X_{t}|c=1/2)=(I_{t})\equiv I$. In particular, we have the specialized Laplace transforms
\begin{equation}\label{1.3}
  \bar{f}_{G_{t}}(u):=\E e^{-uG_{t}}=\bigg(1+\frac{u}{b}\bigg)^{-at}\quad\text{and}\quad\bar{f}_{I_{t}}(u):=\E e^{-uI_{t}}=\exp(2\sqrt{\pi}at(\sqrt{b}-\sqrt{b+u})),
\end{equation}
for $u\in\mathds{C}\setminus(-\infty,-b]$ and any $t\geq0$. We note that tempered stable subordinators and their independent differences have been thoroughly explored in the literature. To name a few, [Schoutens, 2003, $\S$5.3.6] \cite{S1} gave the crucial statistical properties of the class of tempered stable processes and [Rosi\'{n}ski, 2007] \cite{R1} was particularly concerned with their exhibition of stable and Gaussian tendencies in different time frames, while a more comprehensive characterization covering limiting distributions and parameter estimation procedures was presented in [K\"{u}chler and Tappe, 2013] \cite{KT}.

Importantly, the tempered stable subordinator $X$, as well as its two aforementioned special cases, have been widely applied in many fields of science and engineering. For instance, in structural engineering the gamma process has been a popular model for the random occurrence of degradation of certain structural components, and thus conduces to determine optimal inspection and maintenance, since its debut in reliability analysis in [Abdel-Hameed, 1975] \cite{A}. In this respect, [van Noortwijk, 2009] \cite{vN} also provided a survey on the applications of the gamma process. Such degradation phenomena can certainly be modeled by the structurally more general tempered stable subordinator, though using the gamma distribution has its own advantage in terms of simplicity and amenability to parameter estimation. On the other hand, Gaussian mixtures of the tempered stable subordinator are widely applied in mathematical finance. These include the variance gamma processes which were proposed and studied in [Madan and Seneta, 1990] \cite{MS} and [Madan et al, 1998] \cite{MCC} as well as the normal inverse Gaussian processes discussed in [Barndorff-Nielsen, 1997] \cite{B1}, in the context of modeling stock returns and the pricing of financial instruments with short-term large jumps. Remarkably, the tempered stable distribution incorporates positive skewness and excess kurtosis, but yet has all finite moments, which is a very desirable property when it comes to financial modeling.

In this paper we are interested in studying the running average of the tempered stable subordinator, $\tilde{X}\equiv(\tilde{X}_{t})$, as simply given by the following time-scaled integral functional of $X$,
\begin{equation}\label{1.4}
  \tilde{X}_{t}:=\frac{1}{t}\int^{t}_{0}X_{s}\dd s,\quad t>0,
\end{equation}
with $\tilde{X}_{0}=0$, $\PP$-a.s. Equivalently, using It\^{o}'s formula we can write
\begin{equation*}
  \tilde{X}_{t}=\int^{t}_{0}\frac{X_{s}-\tilde{X}_{s}}{s}\dd s,\quad t\geq0,
\end{equation*}
so that it is immediately understood that $\tilde{X}$ has $\PP$-\text{a.s.} nondecreasing and hence nonnegative sample paths, which are also continuous and of finite variation. In this connection, (\ref{1.4}) naturally forms a model for the condition of a structural component subject to stochastic degradation with memory in a continuous manner, which eases analysis of the probability distribution of degradation, while it gives rise to alternative pure jump models in the pricing of financial derivatives.

The remainder of this paper is organized as follows. In Section \ref{sec:2} we provide a complete characterization of the probability distribution of the average-tempered stable subordinator $(\tilde{X}_{t})$ at any fixed time point, which simultaneously leads to a previously unexplored purely discontinuous L\'{e}vy process with infinite activity. For each semi-closed-form formula corresponding asymptotic formulas are derived to make up its deficiency when the argument takes extreme values. Section \ref{sec:3} then serves to illustrate the computational efficiency of the proposed formulas by presenting tables and graphs across different parameter choices. In Section \ref{sec:4}, applications are discussed along two aspects. We establish a degradation model directly based on $\tilde{X}$ as well as construct stock return models as Gaussian mixtures of the newly induced subordinator, whose performance is subsequently evaluated by empirical modeling.

\vspace{0.2in}

\section{Main results}\label{sec:2}

This section provides the main formulas regarding the distribution of $\tilde{X}_{t}$, subject to the parametrization $\{a,b;c\}$ and where $t\geq0$ is treated as generic. We start with the Laplace transform, which generalizes its characteristic function and gives important information on the path properties of a new family of subordinators, and then proceed to deriving the distribution functions.

The Laplace transform of $\tilde{X}_{t}$ takes the following form.

\begin{theorem}[\textbf{Laplace transform}]\label{th:1}
\begin{align}\label{2.1}
  &\bar{f}_{\tilde{X}_{t}}(u)\equiv\bar{f}_{\tilde{X}_{t}}(u|a,b;c):=\E e^{-u\tilde{X}_{t}}=\exp\frac{at\Gf(-c)((b+u)^{c+1}-b^{c}(b+(c+1)u))}{(c+1)u},\\
  &\quad u\in\mathds{C}\setminus(-\infty,-b],\;t\geq0. \nonumber
\end{align}
\end{theorem}

\begin{proof}
A general connection between the characteristic function, namely the Fourier transform of the density function, of a L\'{e}vy process and that of its integral functional is well established. See, e.g., [Xia, 2020, $\S$2] \cite{X}. Thus, according to (\ref{1.1}) we have directly the Laplace exponent
\begin{equation}\label{2.2}
  \log\bar{f}_{\tilde{X}_{t}}(u)=at\Gf(-c)\bigg(\int^{1}_{0}(b+u x)^{c}\dd u-b^{c}\bigg),
\end{equation}
which yields the result after elementary calculations. Then we observe that $\bar{f}_{\tilde{X}_{t}}(u)$ has a removable singularity at the origin and is analytic over the complex $u$-plane except for the principal branch of $(b+u)^{c+1}$.
\end{proof}

In the sequel we shall refer to the family of distributions determined by $\{\bar{f}_{\tilde{X}_{1}}(u;c):c\in(0,1)\}$ as average-tempered stable distributions, denoted by $\mathrm{ATS}(a,b;c)$. Indeed, from Theorem \ref{th:1} it is a straightforward exercise to show that, if $\tilde{X}^{(1)}_{1}\overset{\rm d.}{=}\mathrm{ATS}(a_{1},b;c)$ and $\tilde{X}^{(2)}_{1}\overset{\rm d.}{=}\mathrm{ATS}(a_{2},b;c)$ are independent with $a_{1},a_{2}>0$, then
\begin{equation*}
  \tilde{X}^{(1)}_{1}+\tilde{X}^{(2)}_{1}\overset{\rm d.}{=}\mathrm{ATS}(a_{1}+a_{2},b;c)
\end{equation*}
and that for any scaling factor $\rho>0$ we have
\begin{equation*}
  \rho\tilde{X}_{1}\overset{\rm d.}{=}\mathrm{ATS}\bigg(a\rho^{c},\frac{b}{\rho};c\bigg),
\end{equation*}
which is exactly the same scaling property of the $\mathrm{TS}(a,b;c)$ distribution. Refer to [K\"{u}chler and Tappe, 2013, $\S$4] \cite{KT}.

In fact, the Laplace transform (\ref{2.1}) defines an infinitely divisible distribution on $\mathds{R}_{++}\equiv(0,\infty)$, and we claim that
\begin{equation}\label{2.3}
  \tilde{X}_{t}\overset{\rm d.}{=}\Lambda_{t}\overset{\rm d.}{=}\mathrm{ATS}(at,b;c),\quad t\geq0,
\end{equation}
where $\Lambda\equiv(\Lambda_{t})$ is a L\'{e}vy process having the L\'{e}vy triplet $(\tilde{\alpha},\tilde{\beta},\tilde{\nu})$, as present in the L\'{e}vy-Khintchine representation
\begin{equation}\label{2.4}
  \log\bar{f}_{\tilde{X}_{t}}(u)=t\bigg(-\tilde{\alpha} u+\frac{\tilde{\beta}^{2}u^{2}}{2}+\int_{\mathds{R}\setminus\{0\}}\big(e^{-ux}-1+ux\mathds{1}_{\{|x|<1\}}\big)\tilde{\nu}(\dd x)\bigg),\quad t\geq0,
\end{equation}
wherever $\bar{f}_{\tilde{X}_{t}}$ is well-defined. We present the next corollary which determines the triplet and hence describes the path properties of $\Lambda$.

\begin{corollary}[\textbf{L\'{e}vy triplet}]\label{cr:1}
We have in (\ref{2.4})
\begin{align}\label{2.5}
  \tilde{\alpha}&=\frac{a}{2b^{1-c}}\bigg(\Gf(1-c)+\frac{\Gf(2-c,b)-2\Gf(1-c,b)-b^{2}\Gf(-c,b)}{c+1}\bigg),\\
  \tilde{\beta}&=0, \nonumber\\
  \tilde{\ell}(x)\equiv\frac{\tilde{\nu}(\dd x)}{\dd x}&=\frac{a}{c+1}\bigg(\frac{e^{-bx}}{x^{c+1}}-b^{c+1}\Gf(-c,bx)\bigg)\mathds{1}_{(0,\infty)}(x),\quad x\in\mathds{R}, \nonumber
\end{align}
where $\tilde{\ell}$ is the L\'{e}vy density associated with $\tilde{\nu}$.
\end{corollary}

\begin{proof}
First, from comparing (\ref{2.1}) with (\ref{2.4}) it is immediate that $\tilde{\beta}=0$, which signifies the absence of any Brownian component. Since $\tilde{X}$, and therefore $\Lambda$, is nonnegative, $\PP$-a.s., in the L\'{e}vy-Khintchine representation we must have $\tilde{\nu}((-\infty,0))=0$. Suppose that $\tilde{\nu}$ is non-atomic and, in particular, differentiable over $\mathds{R}_{++}$, admitting that $\tilde{\nu}(\dd x)=\tilde{\ell}(x)\dd x$ for $x>0$, i.e., a L\'{e}vy density exists, which is further assumed to be such that
\begin{equation}\label{2.6}
  \int^{\infty}_{0+}x\tilde{\ell}(x)\dd x<\infty.
\end{equation}
Under the assumption (\ref{2.6}), the integrand in the right-hand side of (\ref{2.4}) is separable and consequently the terms linear in $u$ have to cancel with each other. Hence, the drift component is chosen as
\begin{equation}\label{2.7}
  \tilde{\alpha}=\int^{1}_{0+}x\tilde{\ell}(x)\dd x.
\end{equation}
Finding the L\'{e}vy density comes down to solving a Fredholm integral equation of the first kind, namely
\begin{equation*}
  \int^{\infty}_{0+}(e^{-ux}-1)\tilde{\ell}(x)\dd x=\frac{a\Gf(c)((b+u)^{c+1}-b^{c}(b+(c+1)u))}{(c+1)u},\quad u\in\mathds{C}\setminus(-\infty,-b].
\end{equation*}
To this end after some inspection of the right-hand side of the equation we employ the ansatz
\begin{equation}\label{2.8}
  \tilde{\ell}(x)=\tilde{\ell}_{0}(x)+\tilde{\ell}_{1}(x),
\end{equation}
where $\tilde{\ell}_{0}$ and $\tilde{\ell}_{1}$ respectively solve
\begin{equation}\label{2.9}
  \int^{\infty}_{0+}(e^{-ux}-1)\tilde{\ell}_{0}(x)\dd x=\frac{a\Gf(-c)((b+u)^{c}-b^{c})}{c+1}
\end{equation}
and
\begin{equation}\label{2.10}
  \int^{\infty}_{0+}(e^{-ux}-1)\tilde{\ell}_{1}(x)\dd x=\frac{a\Gf(-c)}{c+1}\bigg(\frac{b(b+u)^{c}-b^{c+1}}{u}-cb^{c}\bigg).
\end{equation}
The case of (\ref{2.9}) is essentially the same as finding the L\'{e}vy density of the tempered stable subordinator $(X_{t})$, except with a scaling factor $1/(c+1)$. Due to (\ref{1.2}) we have immediately
\begin{equation}\label{2.11}
  \tilde{\ell}_{0}(x)=\frac{\ell(x)}{c+1}=\frac{ae^{-bx}}{(c+1)x^{c+1}},\quad x>0.
\end{equation}
For (\ref{2.10}) we note that
\begin{equation*}
  \frac{b(b+u)^{c}-b^{c+1}}{u}=cb^{c}+O(u)
\end{equation*}
near the origin. According to [Bateman, 1954, $\S$5.4.3] \cite{B2}, the following inverse transform has been established,
\begin{equation*}
  \frac{1}{2\pi\ii}\int^{d+\ii\infty}_{d-\ii\infty}\frac{e^{ux}(b+u)^{c}}{u}\dd u=b^{c}\bigg(1-\frac{\Gf(-c,bx)}{\Gf(-c)}\bigg),\quad d>0.
\end{equation*}
Since, by Fubini's theorem, $\int^{\infty}_{0+}\Gf(-c,bx)/\Gf(c)\dd x=-c/b$, we have indeed
\begin{equation}\label{2.12}
  \tilde{\ell}_{1}(x)=-\frac{ab^{c+1}\Gf(-c,bx)}{c+1},\quad x>0.
\end{equation}
Substituting (\ref{2.11}) and (\ref{2.12}) into (\ref{2.8}) we obtain $\tilde{\ell}$. Also, note that integration by parts gives
\begin{equation*}
  \Gf(-c,bx):=\int^{\infty}_{bx}z^{-c-1}e^{-z}\dd z=(bx)^{-c-1}e^{-bx}-(c+1)\int^{\infty}_{bx}z^{-c-2}e^{-z}\dd z,
\end{equation*}
where the last integral is clearly positive, and thus shows that $\ell$ is indeed a positive function. Therefore, the expression for $\tilde{\alpha}$ can be easily deduced via straightforward calculations using (\ref{2.7}), Fubini's theorem, and the recurrence relation $c\Gf(c)=\Gf(c+1)$.

In the same vein, it is easily verifiable that
\begin{equation*}
  \int^{\infty}_{0+}x\tilde{\ell}(x)\dd x=\frac{a\Gf(1-c)}{2b^{1-c}}=\E\tilde{X}_{1}
\end{equation*}
and the assumption (\ref{2.6}) is well satisfied, from which we complete the proof by the uniqueness of the L\'{e}vy triplet.
\end{proof}

From Corollary \ref{cr:1} we see that $\Lambda$, which we name as the average-tempered stable subordinator corresponding to $\tilde{X}$, is a purely discontinuous process. Since $\Gf(-c,bx)=O(x^{-c})$ as $x\searrow0$, the Blumenthal-Getoor index of $\Lambda$ is found to be
\begin{equation*}
  \mathbf{B}(\Lambda)=\inf\bigg\{p>0:\int_{|x|<1}|x|^{p}\tilde{\ell}(x)\dd x<\infty\bigg\}=c=\mathbf{B}(X),
\end{equation*}
which implies that $\Lambda$ is indeed a subordinator with infinite activity, as is $X$. Comparing (\ref{2.5}) with (\ref{1.2}) we see that, under the same parameter values, $X$ tends to have larger jumps relative to $\Lambda$, since $c>0$ and $\Gf(-c,bx)>0$, for any $x>0$, which is ascribed to the effect of averaging.

Applying L\'{e}vy's continuity theorem to (\ref{2.1}) shows that the sample paths of $(\tilde{X}_{t})$ grow like $O(t)$ $\PP$-\text{a.s.} over time, with the strong convergence
\begin{equation*}
  \frac{\tilde{X}_{t}}{t}\overset{\text{a.s.}}{\rightarrow}\frac{a\Gf(1-c)}{2b^{1-c}}>0,\quad\text{as }t\rightarrow\infty,
\end{equation*}
which is also viewable as a consequence of Kolmogorov's strong law of large numbers. On the other hand, noting that
\begin{equation*}
  \lim_{t\searrow0}\log\bar{f}_{\tilde{X}_{t}}(t^{-1/c}u)=-au^{c}\Gf(-c-1),\quad\Re u>0,
\end{equation*}
we have the weak convergence
\begin{equation*}
  \frac{\tilde{X}_{t}}{t^{1/c}}\overset{\text{d.}}{\rightarrow}\mathrm{S}\big(c,(a\Gf(-c-1))^{1/c}\big),\quad\text{as }t\searrow0,
\end{equation*}
where $\mathrm{S}\big(c,(a\Gf(-c-1))^{1/c}\big)$ denotes the one-sided stable distribution with stability parameter $c\in(0,1)$ and scale parameter $(a\Gf(-c-1))^{1/c}>0$. In comparison, as already derived in [Rosi\'{n}ski, 2007, $\S$3] \cite{R1}, we have $X_{t}/t^{1/c}\overset{\text{d.}}{\rightarrow}\mathrm{S}\big(c,(-a\Gf(-c))^{1/c}\big)$. Therefore, we say that the $\mathrm{ATS}(a,b;c)$ distribution tempers a $c$-stable distribution as the $\mathrm{TS}(a,b;c)$ does, but with smaller scales, since $-\Gf(-c)/\Gf(-c-1)=c+1>1$. Furthermore, in the light of Corollary \ref{cr:1}, it can be regarded as a generalized one-sided tempered $c$-stable distribution with a non-monotonic tempering function\footnote{In contrast, we recall that the tempering function associated with the original tempered stable subordinator $X$ is simply $Q(x)=e^{-bx}$, for $x>0$. See, e.g., [K\"{u}chler and Tappe, 2013, $\S$2] \cite{KT}.} given by
\begin{equation*}
  \tilde{Q}(x)=\frac{1}{c+1}\bigg(e^{-bx}-(bx)^{c+1}\Gf(-c,bx)\bigg),\quad x>0,
\end{equation*}
which is positive-valued and builds the connection $\tilde{\nu}(\dd x)=\tilde{Q}(x)\nu_{\rm S}(\dd x)$ for $x>0$, where $\nu_{\rm S}(\dd x)=(a/x^{c+1})\dd x$ is the jump intensity measure of a $c$-stable subordinator. Asymptotic distributions of the $\mathrm{ATS}(a,b;c)$ distribution based on its parameters are summarized in Corollary \ref{cr:2}.

\begin{corollary}[\textbf{Asymptotic distributions}]\label{cr:2}
We have the following weak convergence relations for $\tilde{X}_{1}\overset{\rm d.}{=}\mathrm{ATS}(a,b;c)$.
\begin{align*}
  \mathrm{ATS}(a,b;c)&\rightarrow\delta(0),\quad\text{as }a\searrow0,\\
  \mathrm{ATS}(a,b;c)&\rightarrow\mathrm{S}\big(c,(a\Gf(-c-1))^{1/c}\big),\quad\text{as }b\searrow0,\\
  \mathrm{ATS}(a,b;c)&\rightarrow\mathrm{AG}(a,b),\quad\text{as }c\searrow0,
\end{align*}
where $\mathrm{AG}(a,b)$ denotes the average-gamma distribution with shape parameter $a>0$ and rate parameter $b>0$. Details are given in the proof.
\end{corollary}

\begin{proof}
These are easy consequences from the Laplace transform (\ref{2.1}) and L\'{e}vy's continuity theorem. To be precise, we have
\begin{equation*}
  \lim_{a\searrow0}\log\bar{f}_{\tilde{X}_{1}}(u)=0,\quad\Re u>0,
\end{equation*}
which signifies a degenerate distribution at the origin, and also,
\begin{equation*}
  \lim_{b\searrow0}\log\bar{f}_{\tilde{X}_{1}}(u)=-au^{c}\Gf(-c-1),\quad\Re u>0.
\end{equation*}
The Laplace transform for the $\mathrm{AG}(a,b)$ distribution is given by
\begin{equation}\label{2.13}
  \lim_{c\searrow0}\bar{f}_{\tilde{X}_{1}}(u)=e^{a}\bigg(1+\frac{u}{b}\bigg)^{-a(1+b/u)},\quad\Re u>0,
\end{equation}
which can also be deduced via applying the integral relation (\ref{2.2}) to that of the $\mathrm{G}(a,b)$ distribution in (\ref{1.3}).
\end{proof}

As a remark, (\ref{2.13}) characterizes the distribution of the gamma functional $\int^{1}_{0}G_{s}\dd s$. Another special case is by taking $c=1/2$ in (\ref{2.1}) and we obtain the Laplace transform of the inverse Gaussian functional $\int^{1}_{0}I_{s}\dd s$,
\begin{equation*}
  \bar{f}_{\tilde{X}_{1}}\bigg(u\bigg|c=\frac{1}{2}\bigg)=\exp\frac{4\sqrt{\pi}a(\sqrt{b}(b+3u/2)-(b+u)^{3/2})}{3u},\quad\Re u>0,
\end{equation*}
which corresponds to the average-inverse Gaussian distribution parameterized by $\{a,b\}$, denoted $\mathrm{AIG}(a,b)$.

Now let $f_{\tilde{X}_{t}}(x)\equiv f_{\tilde{X}_{t}}(x|a,b;c):=\PP\{\tilde{X}_{t}\in\dd x\}/\dd x$ be the probability density function of $\tilde{X}_{t}$ for fixed $t>0$. Its connection with the Laplace transform derived in Theorem \ref{th:1} is well understood, namely
\begin{equation*}
  \bar{f}_{\tilde{X}_{t}}(u)=\int^{\infty}_{0}e^{-ux}f_{\tilde{X}_{t}}(x)\dd x,\quad u\in\mathds{C}\setminus(-\infty,-b],
\end{equation*}
and an integral representation is provided in the next theorem.

\begin{theorem}[\textbf{Probability density function}]\label{th:2}
\begin{align}\label{2.14}
  f_{\tilde{X}_{t}}(x)&=\frac{be^{-atb^{c}\Gf(-c)}}{\pi}\int^{1}_{0}\frac{\dd y}{y^{2}}\;\exp\Bigg(-atb^{c}\Gf(-c-1)\Bigg(y+\cos(\pi c)(1-y)\bigg(\frac{1}{y}-1\bigg)^{c}\Bigg)-\frac{bx}{y}\Bigg)\\
  &\qquad\times\sin\Bigg(atb^{c}\sin(\pi c)\Gf(-c-1)(1-y)\bigg(\frac{1}{y}-1\bigg)^{c}\Bigg),\quad x>0,\;t>0. \nonumber
\end{align}
\end{theorem}

\begin{proof}
By the inversion formula
\begin{equation*}
  f_{\tilde{X}_{t}}(x)=\frac{1}{2\pi\ii}\int^{d+\ii\infty}_{d-\ii\infty}e^{ux}\bar{f}_{\tilde{X}_{t}}(u)\dd u,
\end{equation*}
for arbitrary $d>0$. Using Theorem \ref{th:1} we choose by convention a closed rectangular contour oriented counterclockwise that goes along the edges of the interval $(-\infty,-b]$. More specifically, the contour consists of eight pieces,
\begin{equation*}
  C_{R,r}=L_{d,R}\cup L_{-R,R}\cup L_{-R,r}\cup L_{-b,r}\cup L_{\circ}\cup L_{-R,-r}\cup L_{-R,-R}\cup L_{d,-R},
\end{equation*}
with $0<r<R$ and where
\begin{align*}
  &L_{d,R}=\{d+\ii u:u\in[-R,R]\},\quad L_{-R,R}=\{-u+\ii R:u\in[-d,R]\},\\
  &L_{-R,r}=\{-R-\ii u:u\in[-R,-r]\},\quad L_{-b,r}=\{u+\ii r:u\in[-R,-b]\},\\
  &L_{\circ}=\bigg\{-b+re^{-\ii\theta}:\theta\in\bigg[-\frac{\pi}{2},\frac{\pi}{2}\bigg]\bigg\},\quad L_{-R,-r}=\{-u-\ii r:u\in[b,R]\},\\
  &L_{-R,-R}=\{-R-\ii u:u\in[0,R]\},\quad L_{d,-R}=\{u-\ii R:u\in[-R,d]\}.
\end{align*}
Since $\big|\bar{f}_{\tilde{X}_{t}}(u)\big|$ vanishes in the limit as $u\rightarrow\infty$, sending $R\rightarrow\infty$ and $r\searrow0$ one can show that the integrals along $L_{-R,R}$, $L_{-R,r}$, $L_{\circ}$, $L_{-R,-R}$ and $L_{d,-R}$ all vanish. Consequently, by Cauchy's integral theorem $\int_{C_{R,r}}=0$ and we obtain
\begin{equation*}
  f_{\tilde{X}_{t}}(x)=-\frac{1}{2\pi\ii}\lim_{\substack{R\rightarrow\infty,\\r\searrow0}}\bigg(\int_{L_{-b,r}}+\int_{L_{-R,-r}}\bigg) e^{ux}\bar{f}_{\tilde{X}_{t}}(u)\dd u
\end{equation*}
and after parameterizing the two lines and plugging in (\ref{2.1})
\begin{equation}\label{2.15}
  f_{\tilde{X}_{t}}(x)=\frac{1}{\pi}\int^{\infty}_{b}e^{-ux}\big(-\Im\bar{f}_{\tilde{X}_{t}}(-u)\big)\dd u,
\end{equation}
in which
\begin{equation*}
  \Im\bar{f}_{\tilde{X}_{t}}(-u)=e^{\Re\log\bar{f}_{\tilde{X}_{t}}(-u)}\sin\Im\log\bar{f}_{\tilde{X}_{t}}(-u),\quad u>b.
\end{equation*}
Using that $(b-u)^{c+1}=-(u-b)^{c+1}e^{\ii\pi c}$, (\ref{2.15}) gives
\begin{align}\label{2.16}
  f_{\tilde{X}_{t}}(x)&=\frac{e^{-atb^{c}\Gf(-c)}}{\pi}\int^{\infty}_{b}\dd u\;\exp\bigg(\frac{at\Gf(-c)(b^{c+1}+\cos(\pi c)(u-b)^{c+1})}{(c+1)u}-ux\bigg)\\
  &\qquad\times\sin\frac{-at\sin(\pi c)\Gf(-c)(u-b)^{c+1}}{(c+1)u}. \nonumber
\end{align}
The proof is completed by applying the substitution $y=b/u$ to (\ref{2.16}) and using again the recurrence relation of $\Gf(\cdot)$, thus leading to a proper definite integral.
\end{proof}

By the observation that
\begin{equation*}
  \frac{1}{y^{2}}\exp\Bigg(-atb^{c}\Gf(-c-1)(1-y)\bigg(\frac{1}{y}-1\bigg)^{c}-\frac{bx}{y}\Bigg)\leq\frac{e^{1-bx/y}}{y^{2}}, \quad y\in(0,1),
\end{equation*}
we have that the integrand in the representation (\ref{2.14}) is bounded over the unit interval and hence the integral is well defined for every $x>0$, which verifies that the distribution of $\tilde{X}_{t}$ for $t>0$ implied by Theorem 2 is absolutely continuous with respect to the Lebesgue measure. In the following corollary further comments are given on the structural properties of the probability density function along this aspect.

\begin{corollary}[\textbf{Smoothness and unimodality}]\label{cr:3}
The probability density function in (\ref{2.14}) is infinitely smooth and unimodal on $\mathds{R}_{++}$.
\end{corollary}

\begin{proof}
The infinite smoothness of the probability density function follows from that of the integrand inside (\ref{2.14}), which is equivalent to (\ref{2.16}). As for unimodality, we recall the representation (\ref{2.4}) of the Laplace exponent. Since $(e^{-ux}-1+ux/(1+x^{2}))/x=x/2+O(x)$ as $x\searrow0$ for $\Re u>0$, the Laplace exponent of $\tilde{X}_{1}$ admits that
\begin{equation*}
  \log\bar{f}_{\tilde{X}_{1}}(u)=-\tilde{\eta}u+\int^{\infty}_{0+}\bigg(e^{-ux}-1+\frac{ux}{1+x^{2}}\bigg)\tilde{\ell}(x)\dd x,\quad\Re u>0,
\end{equation*}
where
\begin{equation*}
  \tilde{\eta}=\int^{\infty}_{0+}\frac{x\tilde{\ell}(x)}{1+x^{2}}\dd x<\infty.
\end{equation*}
From Corollary \ref{cr:1}, we already know that $x\tilde{\ell}(x)>0$ for $x>0$, $\int^{1}_{0+}x^{2}\tilde{\ell}(x)\dd x<\infty$, and also $\int^{\infty}_{1}\tilde{\ell}(x)\dd x<\infty$, and then observe that
\begin{equation*}
  \frac{\dd(x\tilde{\ell}(x))}{\dd x}=-\frac{acx^{-c-1}e^{-bx}+ab^{c+1}\Gf(-c,bx)}{c+1}<0,\quad x>0,
\end{equation*}
so that $x\tilde{\ell}(x)$ is strictly decreasing in $x>0$. Therefore, we claim that the $\mathrm{ATS}(a,b;c)$ distribution belongs to class L in the sense of [Yamazato, 1978, $\S$1] \cite{Y1}, where it was also shown that all class-L distributions must be unimodal.
\end{proof}

Efficient implementation of the integral representation (\ref{2.14}) can be carried out by employing the Gauss quadrature rule proposed in [Golub and Welsch, 1969] \cite{GW} and its Kronrod extension discussed in [Laurie, 1997] \cite{L1}. Indeed, it is easy to notice that, despite the highly oscillatory feature of the sine function near the origin, the integrand tends to 0 exponentially fast as $y\searrow0$, and it attains a zero of the sine function as $y\nearrow1$, uniformly in $x$. On the other hand, we find it quite knotty to obtain a single series representation of the probability density function, due to the general exponential form of $\bar{f}_{\tilde{X}_{t}}(u)$. Also, according to Corollary \ref{cr:3}, the mode of $\tilde{X}_{t}$ is uniquely identified as
\begin{align*}
  &\Bigg\{x>0:\int^{1}_{0}\frac{\dd y}{y^{3}}\;\exp\Bigg(-atb^{c}\Gf(-c-1)\Bigg(y+\cos(\pi c)(1-y)\bigg(\frac{1}{y}-1\bigg)^{c}\Bigg)-\frac{bx}{y}\Bigg)\\
  &\qquad\times\sin\Bigg(atb^{c}\sin(\pi c)\Gf(-c-1)(1-y)\bigg(\frac{1}{y}-1\bigg)^{c}\Bigg)=0\Bigg\},\quad t>0,
\end{align*}
which can be solved numerically by minimizing the magnitude of the integral over the domain of $x$.

Moreover, it is possible to recover the L\'{e}vy density of $\Lambda$ stated in Corollary \ref{cr:1} making use of the probability density function, preferably (\ref{2.16}). This heavily relies on the asymptotic relation that
\begin{equation*}
  \lim_{t\searrow0}\frac{\E\varphi(\tilde{X}_{t})}{t}\equiv\lim_{t\searrow0}\frac{1}{t}\int^{\infty}_{0}\varphi(x)f_{\tilde{X}_{t}}(x)\dd x=\int^{\infty}_{0}\varphi(x)\tilde{\ell}(x)\dd x,
\end{equation*}
for any bounded and sufficiently smooth test function $\varphi:\mathds{R}_{++}\mapsto\mathds{R}$ in the presence of the distributional equivalence (\ref{2.3}), which implies that
\begin{equation*}
  \tilde{\ell}(x)=\lim_{t\searrow0}\frac{f_{\tilde{X}_{t}}(x)}{t}
\end{equation*}
by the absolute continuity of the distribution. The limit can be evaluated easily after expanding the integrand of (\ref{2.16}) around $t=0$, which then yields upon integration (\ref{2.5}).

Likewise, we have an integral formula for the cumulative distribution function,
\begin{equation*}
  F_{\tilde{X}_{t}}(x)\equiv\PP\{\tilde{X}_{t}\leq x\}=\int^{x}_{0}f_{\tilde{X}_{t}}(z)\dd z,\quad x>0,\;t>0.
\end{equation*}

\begin{theorem}[\textbf{Cumulative distribution function}]\label{th:3}
\begin{align}\label{2.17}
  F_{\tilde{X}_{t}}(x)&=1-\frac{e^{-atb^{c}\Gf(-c)}}{\pi}\int^{1}_{0}\frac{\dd y}{y}\;\exp\Bigg(-atb^{c}\Gf(-c-1)\Bigg(y+\cos(\pi c)(1-y)\bigg(\frac{1}{y}-1\bigg)^{c}\Bigg)-\frac{bx}{y}\Bigg)\\
  &\qquad\times\sin\Bigg(atb^{c}\sin(\pi c)\Gf(-c-1)(1-y)\bigg(\frac{1}{y}-1\bigg)^{c}\Bigg),\quad x>0,\;t>0. \nonumber
\end{align}
\end{theorem}

\begin{proof}
This is an immediate consequence from applying Theorem \ref{th:2}, Fubini's theorem, as well as the fact that $\lim_{x\rightarrow\infty}F_{\tilde{X}_{t}}(x)=1$.
\end{proof}

As special cases, the probability density function and cumulative distribution function of the $\mathrm{AG}(at,b)$ distribution can be immediately identified from Theorems \ref{th:2} and \ref{th:3} as follows,
\begin{equation}\label{2.18}
  f_{\tilde{X}_{t}}(x|c\searrow0)=\frac{be^{at}}{\pi}\int^{1}_{0}\dd y\;\bigg(\frac{1}{y}-1\bigg)^{-at(1-y)}\frac{\sin(\pi at(1-y))e^{-bx/y}}{y^{2}}
\end{equation}
and
\begin{equation*}
  F_{\tilde{X}_{t}}(x|c\searrow0)=1-\frac{e^{at}}{\pi}\int^{1}_{0}\dd y\;\bigg(\frac{1}{y}-1\bigg)^{-at(1-y)}\frac{\sin(\pi at(1-y))e^{-bx/y}}{y},
\end{equation*}
while those of the $\mathrm{AIG}(at,b)$ distribution can be similarly written
\begin{equation*}
  f_{\tilde{X}_{t}}\bigg(x\bigg|c=\frac{1}{2}\bigg)=\frac{be^{2at\sqrt{\pi b}}}{\pi}\int^{1}_{0}\frac{\dd y}{y^{2}}\;\exp\bigg(-\frac{4at\sqrt{\pi b}y}{3}-\frac{bx}{y}\bigg)\sin\bigg(\frac{4at\sqrt{\pi b}(1-y)}{3}\sqrt{\frac{1}{y}-1}\bigg)
\end{equation*}
and
\begin{equation*}
  F_{\tilde{X}_{t}}\bigg(x\bigg|c=\frac{1}{2}\bigg)=1-\frac{e^{2at\sqrt{\pi b}}}{\pi}\int^{1}_{0}\frac{\dd y}{y}\;\exp\bigg(-\frac{4at\sqrt{\pi b}y}{3}-\frac{bx}{y}\bigg)\sin\bigg(\frac{4at\sqrt{\pi b}(1-y)}{3}\sqrt{\frac{1}{y}-1}\bigg),
\end{equation*}
for $x>0$ and $t>0$.

Some results are next established for the tail behaviors of the probability density function and the cumulative distribution function represented in Theorem \ref{th:2} and Theorem \ref{th:3}, aiming at making up their deficiency when their argument $x>0$ takes extreme values.

\begin{theorem}[\textbf{Tail behaviors of probability density function}]\label{th:4}
We have for $t>0$
\begin{equation}\label{2.19}
  f_{\tilde{X}_{t}}(x)=\frac{ate^{atb^{c}c\Gf(-c-1)-bx}}{bx^{c+2}}(1+O(x^{-c-1})),\quad\text{as }x\rightarrow\infty,
\end{equation}
and
\begin{align}\label{2.20}
  f_{\tilde{X}_{t}}(x)&=\frac{be^{-atb^{c}\Gf(-c)}}{\pi}\int^{1}_{0}\frac{\dd y}{y^{2}}\;\exp\Bigg(-atb^{c}\cos(\pi c)\Gf(-c-1)\bigg(\frac{1}{y}-1\bigg)^{c}-\frac{bx}{y}\Bigg)\\
  &\qquad\times\sin\Bigg(atb^{c}\sin(\pi c)\Gf(-c-1)\bigg(\frac{1}{y}-1\bigg)^{c}\Bigg)(1+O(x^{1-c})),\quad\text{as }x\searrow0. \nonumber
\end{align}
\end{theorem}

\begin{proof}
First we work with (\ref{2.16}), where the substitution $y=u-b$ leads to the alternative representation
\begin{align}\label{2.21}
  f_{\tilde{X}_{t}}(x)&=\frac{e^{-atb^{c}\Gf(-c)}}{\pi}\int^{\infty}_{0}\dd y\;\exp\bigg(\frac{at\Gf(-c)(b^{c+1}+\cos(\pi c)y^{c+1})}{(c+1)(b+y)}-(b+y)x\bigg)\\
  &\qquad\times\sin\frac{-at\Gf(-c)\sin(\pi c)y^{c+1}}{(c+1)(b+y)},\quad x>0. \nonumber
\end{align}
Since $c<1$, we take the Taylor expansion of the exponent and the sine part around $y=0$ and obtain respectively
\begin{equation}\label{2.22}
  \frac{at\Gf(-c)(b^{c+1}+\cos(\pi c)y^{c+1})}{(c+1)(b+y)}=-atb^{c}\Gf(-c-1)+O(y^{c})
\end{equation}
and
\begin{equation}\label{2.23}
  \sin\frac{-at\sin(\pi c)\Gf(-c)y^{c+1}}{(c+1)(b+y)}=-\frac{at\sin(\pi c)\Gf(-c)y^{c+1}}{b(c+1)}+O(y^{c+2}).
\end{equation}
Multiplying (\ref{2.22}) and (\ref{2.23}) together yields with (\ref{2.21})
\begin{align*}
  \int^{\infty}_{0}&=-\frac{at\sin(\pi c)\Gf(-c)e^{-atb^{c}\Gf(-c-1)-bx}}{b(c+1)}\int^{\infty}_{0}\dd y\;y^{c+1}e^{-xy}(1+O(y^{c+1}))\\
  &=-\frac{at\sin(\pi c)\Gf(-c)e^{-atb^{c}\Gf(-c-1)-bx}}{b(c+1)}\times\frac{\Gf(c+2)}{x^{c+2}}(1+O(x^{-c-1}))\\
  &=\frac{\pi ate^{-atb^{c}\Gf(-c-1)-bx}}{bx^{c+2}}(1+O(x^{-c-1})),
\end{align*}
where the third equality uses Euler's reflection formula for $\Gf(\cdot)$. Putting the last result into (\ref{2.21}) with minor simplifications leads to (\ref{2.19}).

For the left-tail behavior we take the asymptotic expansion of the Laplace transform (\ref{2.1})
\begin{align*}
  \bar{f}_{\tilde{X}_{t}}(u)&=\exp(-atb^{c}\Gf(-c)-(b+u)^{c}(at\Gf(-c-1)+O(u^{-1})))\\
  &=e^{-ab^{c}t\Gf(-c)}\exp(-at\Gf(-c-1)(b+u)^{c})(1+O(u^{c-1})),\quad\text{as }\Re u\rightarrow\infty,
\end{align*}
so that the asymptotic behavior of the density function near the origin is given by the Laplace inverse
\begin{equation}\label{2.24}
  f_{\tilde{X}_{t}}(x)=\frac{e^{-ab^{c}t\Gf(-c)}}{2\pi\ii}\int^{d+\ii\infty}_{d-\ii\infty}\dd u\;e^{ux-at\Gf(-c-1)(b+u)^{c}}(1+O(x^{1-c})),\quad d>0,
\end{equation}
where note that $at\Gf(-c-1)>0$. Unfortunately, the inverse (\ref{2.24}) cannot be evaluated explicitly for general $c\in(0,1)$. Towards this end we apply exactly the same contour deformation argument as in the proof of Theorem \ref{th:2}, which results in the real integral
\begin{align*}
  \frac{1}{2\pi\ii}\int^{d+\ii\infty}_{d-\ii\infty}\dd u\;e^{ux-at\Gf(-c-1)(b+u)^{c}}&=\frac{1}{\pi}\int^{\infty}_{b}\dd u\;e^{-ux-at\cos(\pi c)\Gf(-c-1)(u-b)^{c}}\\
  &\qquad\times\sin(at\sin(\pi c)\Gf(-c-1)(u-b)^{c}),
\end{align*}
and subsequently the substitution $y=b/u$ to obtain (\ref{2.20}).
\end{proof}

The asymptotic estimate in (\ref{2.19}) indicates that the right tail of $f_{\tilde{X}_{t}}(x)$ has power-adjusted exponential decay, which is heavier for smaller values of $c$ and is the same as the right tail of the tempered stable density function, $f_{X_{t}}(x)$, up to a positive scaling factor. The specialized estimates for the inverse Gaussian and gamma cases can be immediately obtained.
\begin{equation*}
  f_{\tilde{X}_{t}}\bigg(x\bigg|c=\frac{1}{2}\bigg)=\frac{ate^{2at\sqrt{\pi b}/3-bx}}{bx^{5/2}}(1+O(x^{-3/2})),\quad\text{as }x\rightarrow\infty,
\end{equation*}
and
\begin{equation*}
  f_{\tilde{X}_{t}}(x|c\searrow0)=\frac{ate^{at-bx}}{bx^{2}}(1+O(x^{-1})),\quad\text{as }x\rightarrow\infty.
\end{equation*}

On the basis of (\ref{2.24}), by the initial-value theorem, since $\lim_{u\rightarrow\infty}ue^{-at\Gf(-c-1)(b+u)^{c}}=0$, we know that the left tail of $f_{\tilde{X}_{t}}(x)$ vanishes eventually for all $c\in(0,1)$. On the other hand, noting that the Laplace inverse of $e^{-at\Gf(-c-1)(b+u)}$ in $u$ is a point mass, it can be implied that the closer $c$ is to 1 the flatter the left tail of $f_{\tilde{X}_{t}}(x)$ against $x$ becomes. In the inverse Gaussian case with $c=1/2$, (\ref{2.20}) actually gives an explicit estimate,
\begin{equation*}
  \bar{f}_{\tilde{X}_{t}}\bigg(u\bigg|c=\frac{1}{2}\bigg)=\frac{2at}{3x^{3/2}}\exp\bigg(2at\sqrt{\pi b}-bx-\frac{4\pi at^{2}}{9x}\bigg)(1+O(x^{1/2})),\quad\text{as }x\searrow0.
\end{equation*}
due to [Bateman, 1954, $\S$5.6.1] \cite{B2}. In the gamma case with $c\searrow0$, explicit estimates are also readily available, whereas the left tail does not have to vanish, which depends further on the factor $at>0$. We present the next corollary for this limiting case.

\begin{corollary}[\textbf{Left tail behavior of average-gamma probability density function}]\label{cr:4}
\begin{equation*}
  f_{\tilde{X}_{t}}(x|c\searrow0)=
  \begin{cases}
    \displaystyle \frac{(eb)^{at}x^{at-1}}{\Gf(at)}(1+O(x)),&\text{ if }at\neq1,\\
    \displaystyle eb(1+bx(\gamma-1+\log(bx))+O(x^{2})),&\text{ if }at=1,\\
  \end{cases}\quad\text{as }x\searrow0,
\end{equation*}
where $\gamma$ is the Euler-Mascheroni constant.
\end{corollary}

\begin{proof}
Instead of working with limits, we consult (\ref{2.13}) in the proof of Corollary \ref{cr:2} and directly obtain the estimate
\begin{equation*}
  \bar{f}_{\tilde{X}_{t}}(u|c\searrow0)=\bigg(\frac{eb}{u}\bigg)^{at}(1+O(u^{-1})),\quad\text{as }\Re u\rightarrow\infty,
\end{equation*}
from where the inversion formula gives
\begin{equation}\label{2.25}
  f_{\tilde{X}_{t}}(x|c\searrow0)=\frac{(eb)^{at}x^{at-1}}{\Gf(at)}(1+O(x)),\quad\text{as }x\searrow0.
\end{equation}
Clearly, (\ref{2.25}) implies
\begin{equation*}
  \lim_{x\searrow0}f_{\tilde{X}_{t}}(x|c\searrow0)=\infty\mathds{1}_{(0,1)}(at)+eb\mathds{1}_{\{1\}}(at)
\end{equation*}
and hence if $at=1$ a finer estimate is needed. We then consider in this case the second-order expansion that
\begin{equation*}
  \bar{f}_{\tilde{X}_{t}}(u|c\searrow0,at=1)=\frac{eb}{u}\bigg(1-\frac{b(1-\log b+\log u)}{u}+O(u^{-2})\bigg).
\end{equation*}
For the Laplace inverse of $\log u/u^{2}$ one can refer to [Bateman, 1954, $\S$5.7.2] \cite{B2} and as a consequence,
\begin{equation*}
  f_{\tilde{X}_{t}}(x|c\searrow0,at=1)=eb(1+bx(\gamma-1+\log(bx))+O(x^{2})),\quad\text{as }x\searrow0.
\end{equation*}
\end{proof}

Clearly, Corollary \ref{cr:4} shows that, in the gamma case, the left tail of $f_{\tilde{X}_{t}}(x|c\searrow0)$ vanishes if and only if $at>1$, while for $at=1$ it goes to $eb>0$ and it explodes if $at<1$. Since $f_{\tilde{X}_{t}}(0|c\searrow0,at=1)=eb$, according to (\ref{2.18}) we have also proven the curious identity
\begin{equation*}
  \int^{1}_{0}\dd y\;\bigg(\frac{1}{y}-1\bigg)^{y-1}\frac{\sin(\pi(1-y))}{y^{2}}=\pi.
\end{equation*}

In the same vein, estimates can be deduced for the tail behaviors of the cumulative distribution function.

\begin{theorem}[\textbf{Tail behaviors of cumulative distribution function}]\label{th:5}
We have for $t>0$
\begin{equation}\label{2.26}
  1-F_{\tilde{X}_{t}}(x)=\frac{ate^{atb^{c}\Gf(-c-1)-bx}}{b^{2}x^{c+2}}(1+O(x^{-c-1})),\quad\text{as }x\rightarrow\infty.
\end{equation}
and
\begin{align}\label{2.27}
  F_{\tilde{X}_{t}}(x)&=e^{atb^{c}c\Gf(-c-1)}-\frac{e^{-atb^{c}\Gf(-c)}}{\pi}\int^{1}_{0}\frac{\dd y}{y}\;\exp\Bigg(-atb^{c}\cos(\pi c)\Gf(-c-1)\bigg(\frac{1}{y}-1\bigg)^{c}-\frac{bx}{y}\Bigg)\\
  &\qquad\times\sin\Bigg(atb^{c}\sin(\pi c)\Gf(-c-1)\bigg(\frac{1}{y}-1\bigg)^{c}\Bigg)(1+O(x^{c-1})),\quad\text{as }x\searrow0. \nonumber
\end{align}
\end{theorem}

\begin{proof}
The proof of (\ref{2.26}) is similar to that of (\ref{2.19}), based on integration of (\ref{2.21}) in $x$ and the fact that $\lim_{x\rightarrow\infty}F_{\tilde{X}_{t}}(x)=1$. On the other hand, deriving (\ref{2.27}) exploits the integration property of Laplace transforms that
\begin{equation*}
  \bar{F}_{\tilde{X}_{t}}(u):=\int^{\infty}_{0}e^{-ux}F_{\tilde{X}_{t}}(x)\dd x=\frac{\bar{f}_{\tilde{X}_{t}}(u)}{u},\quad u\in\mathds{C}\setminus(-\infty,-b],
\end{equation*}
and the proof of (\ref{2.20}) directly applies except for the simple pole of $e^{-at\Gf(-c-1)(u-b)^{c}}/u$ at $u=0$.
\end{proof}

We proceed to the moments of $\tilde{X}_{t}$, where $t\geq0$ is fixed. Denote by $M_{\tilde{G}_{t}}(n)$ its $n$th moment. The presence of Theorem \ref{th:1} guarantees existence of all orders of moments, which satisfy
\begin{equation*}
  M_{\tilde{X}_{t}}(n)\equiv\E\tilde{X}^{n}_{t}=\frac{(-1)^{n}\dd^{n}\bar{f}_{\tilde{X}_{t}}(u)}{\dd u^{n}}\bigg|_{u=0},
\end{equation*}
and a recurrence formula can be given.

\begin{theorem}[\textbf{Moments}]\label{th:6}
\begin{equation}\label{2.28}
  M_{\tilde{X}_{t}}(n)=
  \begin{cases}
    \displaystyle 1,&\text{ if }n=0,\\
    \displaystyle at\sum^{n}_{k=0}\binom{n}{k}\frac{\Gf(k-c+1)}{(k+2)b^{k-c+1}}M_{\tilde{X}_{t}}(n-k),&\text{ if }n\in\mathds{N}_{++},
  \end{cases}\quad t\geq0.
\end{equation}
\end{theorem}

\begin{proof}
First we determine the cumulants of $\tilde{X}_{t}$, denoted $C_{\tilde{X}_{t}}(n)$, for $n\in\mathds{N}$, which are the series coefficients inside the Laplace exponent
\begin{equation}\label{2.29}
  \log\bar{f}_{\tilde{X}_{t}}(u)=\sum^{\infty}_{n=0}\frac{C_{\tilde{X}_{t}}(n)(-u)^{n}}{n!}.
\end{equation}
According to Theorem \ref{th:1}, the Laplace exponent admits a simple binomial series expansion about the origin, which gives after some rearrangement
\begin{equation}\label{2.30}
  \log\bar{f}_{\tilde{X}_{t}}(u)=at\Gf(-c)\sum^{\infty}_{n=2}\binom{c+1}{n}\frac{b^{c+1-n}u^{n-1}}{c+1}.
\end{equation}
Comparing (\ref{2.29}) with (\ref{2.30}) and matching coefficients we obtain the cumulant formula
\begin{equation}\label{2.31}
  C_{\tilde{X}_{t}}(n)=at\Gf(-c)\binom{c+1}{n+1}\frac{(-1)^{n}n!}{(c+1)b^{n-c}}=\frac{at(-1)^{n}\Gf(-c)\Gf(c+1)}{(n+1)b^{n-c}\Gf(c-n+1)},\quad n\in\mathds{N}_{++},
\end{equation}
with $C_{\tilde{X}_{t}}(0)=0$, recalling that $\bar{f}_{\tilde{X}_{t}}(u)$ has a removable singularity at the origin. For $n$ integer-valued, we use again Euler's reflection formula $n$ times to simplify (\ref{2.31}) into
\begin{equation}\label{2.32}
  C_{\tilde{X}_{t}}(n)=\frac{at\Gf(n-c)}{(n+1)b^{n-c}},\quad n\in\mathds{N}_{++}.
\end{equation}

Then, since the moments are identified as the series coefficients in
\begin{equation*}
  \bar{f}_{\tilde{X}_{t}}(u)=\sum^{\infty}_{n=0}\frac{M_{\tilde{X}_{t}}(n)(-u)^{n}}{n!},
\end{equation*}
in line with the cumulant formula (\ref{2.32}) we employ the exponential formula in combinatorics (see, e.g., [Stanley, 1999, $\S$1.1] \cite{S2}) and consequently obtain the Bell polynomials
\begin{equation*}
  M_{\tilde{X}_{t}}(n)=\mathrm{Bell}(C_{\tilde{X}_{t}}(k)|k\in\mathds{N}_{++}\cap[1,n]),\quad n\in\mathds{N},
\end{equation*}
which is also recurrently given by
\begin{equation}\label{2.33}
  M_{\tilde{X}_{t}}(0)=1\rightsquigarrow M_{\tilde{X}_{t}}(n)=\sum^{n}_{k=0}\binom{n}{k}C_{\tilde{X}_{t}}(k+1)M_{\tilde{X}_{t}}(n-k),\quad n\in\mathds{N}_{++}.
\end{equation}
Substituting (\ref{2.32}) into (\ref{2.33}) completes the proof.
\end{proof}

Alternatively, one may also exploit Theorem \ref{th:2} together with the fundamental relation
\begin{equation}\label{2.34}
  M_{\tilde{X}_{t}}(n)=\int^{\infty}_{0}x^{n}f_{\tilde{X}_{t}}(x)\dd x,\quad n\in\mathds{N}
\end{equation}
to obtain an equivalent numerical integral formula for the moments, namely
\begin{align*}
  M_{\tilde{X}_{t}}(n)&=\frac{e^{-atb^{c}\Gf(-c)}n!}{\pi b^{n}}\int^{1}_{0}\dd y\;y^{n-1}\exp\Bigg(-atb^{c}\Gf(-c-1)\Bigg(y+\cos(\pi c)(1-y)\bigg(\frac{1}{y}-1\bigg)^{c}\Bigg)\Bigg)\\
  &\qquad\times\sin\Bigg(atb^{c}\sin(\pi c)\Gf(-c-1)(1-y)\bigg(\frac{1}{y}-1\bigg)^{c}\Bigg),\quad x>0,
\end{align*}
for any $t>0$. We also highlight that taking $c$ to be 0 in (\ref{2.28}) is perfectly legitimate, which automatically gives the moment formula for the $\mathrm{AG}(at,b)$ distribution.

Regarding the four crucial statistics, the mean and variance of $\tilde{X}_{t}$ are simply $C_{\tilde{X}_{t}}(1)$ and $C_{\tilde{X}_{t}}(2)$, respectively, while its skewness and excess kurtosis are respectively calculated as $C_{\tilde{X}_{t}}(3)\big/C^{3/2}_{\tilde{X}_{t}}(2)$ and $C_{\tilde{X}_{t}}(4)\big/C^{2}_{\tilde{X}_{t}}(2)$. Specifically, we obtain
\begin{align}\label{2.35}
  \E\tilde{X}_{t}&=\frac{at\Gf(1-c)}{2b^{1-c}}=\frac{1}{2}\E X_{t},\\
  \Var\tilde{X}_{t}&=\frac{at\Gf(2-c)}{3b^{2-c}}=\frac{1}{3}\Var X_{t}, \nonumber\\
  \Skew\tilde{X}_{t}&=\frac{3\sqrt{3}(2-c)\Gf(3-c)}{4\sqrt{atb^{c}\Gf(2-c)^{3}}}=\frac{3\sqrt{3}}{4}\Skew X_{t}, \nonumber\\
  \EKurt\tilde{X}_{t}&=\frac{9\Gf(4-c)}{5atb^{c}\Gf(2-c)^{2}}=\frac{9}{5}\EKurt X_{t}. \nonumber
\end{align}
This indicates that the $\mathrm{ATS}(a,b;c)$ distribution is more asymmetric and leptokurtic than the $\mathrm{TS}(a,b;c)$ distribution. At the same time, one can easily obtain the corresponding statistics for the $\mathrm{AG}(a,b)$ and $\mathrm{AI}(a,b)$ distributions, by sending $c\searrow0$ and taking $c=1/2$ respectively in (\ref{2.35}).

In addition, for arbitrary times $t,v>0$, using (\ref{2.35}) and the fact that $\E(X_{t}X_{v})=(t\wedge v)\Var X_{1}+(\E X_{1})^{2}tv$ based on the L\'{e}vy properties of $X$ the following covariance function can be given from (\ref{1.4}),
\begin{equation*}
  \Cov(\tilde{X}_{t},\tilde{X}_{v})=\frac{(t\wedge v)(3(t\vee v)-t\wedge v)\Var X_{1}}{6(t\vee v)}=\frac{a\Gf(2-c)(3(t\vee v)-t\wedge v)}{6b^{2-c}}\bigg(\frac{v}{t}\wedge\frac{t}{v}\bigg),
\end{equation*}
which is different from that of $\Lambda$. In fact, since $\Lambda$ is itself a L\'{e}vy process we have immediately $\Cov(\Lambda_{t},\Lambda_{v})=\Var\Lambda_{t\wedge v}=a\Gf(2-c)(t\wedge v)/(3b^{2-c})$.

The next corollary is presented to describe the asymptotic behavior of the moments for large orders.

\begin{corollary}[\textbf{Large-order behavior of moments}]\label{cr:5}
\begin{equation*}
  M_{\tilde{X}_{t}}(n)=\frac{ate^{atb^{c}c\Gf(-c-1)}\Gf(n-c-1)}{b^{n-c}}(1+O(n^{-c-1})),\quad\text{as }n\rightarrow\infty.
\end{equation*}
\end{corollary}

\begin{proof}
Substituting (\ref{2.19}) into (\ref{2.34}) it follows that
\begin{align*}
  M_{\tilde{X}_{t}}(n)&=\frac{ate^{atb^{c}c\Gf(-c-1)}}{b}\int^{\infty}_{0}\dd x\;x^{n-c-2}e^{-bx}(1+O(x^{-c-1}))\\
  &=\frac{ate^{atb^{c}c\Gf(-c-1)}}{b}(b^{c-n+1}\Gf(n-c-1)+O(\Gf(n-2c-2))),
\end{align*}
for $n$ large. Then Stirling's formula applied to $\Gf(\cdot)$ tells us that
\begin{equation*}
  O\bigg(\frac{\Gf(n-2c-2)}{\Gf(n-c-1)}\bigg)=O\bigg(\frac{((n-2c-3)/e)^{n-2c-3}}{((n-c-2)/e)^{n-c-2}}\bigg)=O(n^{-c-1}),\quad\text{as }n\rightarrow\infty,
\end{equation*}
which yields the desired result.
\end{proof}

\vspace{0.2in}

\section{Numerical experiments}\label{sec:3}

In this section we give some numerical examples to illustrate the validity and efficiency of selected formulas for the $\mathrm{ATS}(at,b;c)$ distribution, namely that of $\tilde{X}_{t}$ for $t>0$. Since $b$ places only a scaling effect and $t$ is a multiplier of $a$, we will fix $b=t=1$, while making comparison across four choices of the shape parameter, with $a=1/2,1,3/2,2$, and two choices of the stability parameter $c\searrow0$ and $c=1/2$, which actually correspond to the $\mathrm{AG}(at,b)$ and $\mathrm{AIG}(at,b)$ distributions, respectively.

First, Figure \ref{fig:1} applies Theorem \ref{th:1} to plot for each choice of $\{a;c\}$ the moment generating function of $\tilde{X}_{1}$, $\bar{f}_{\tilde{X}_{1}}(-u)$, for $u<b=1$.

\begin{figure}[H]
  \centering
  \begin{minipage}[c]{0.49\linewidth}
  \centering
  \includegraphics[width=2.5in]{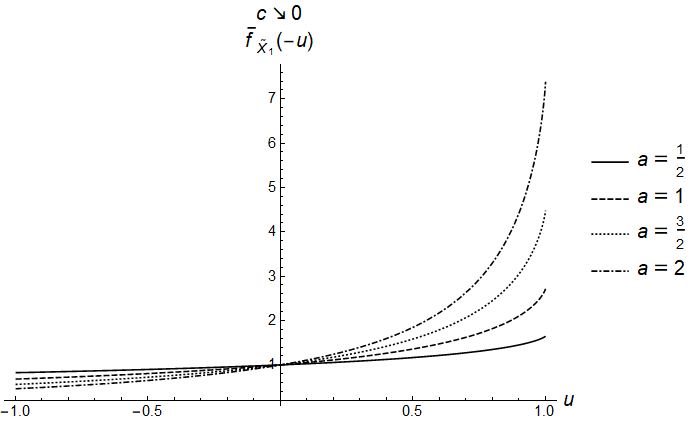}
  \end{minipage}
  \begin{minipage}[c]{0.49\linewidth}
  \centering
  \includegraphics[width=2.5in]{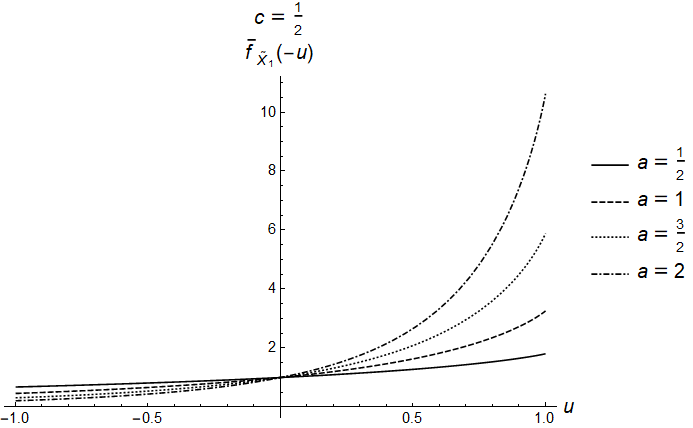}
  \end{minipage}
  \caption{Moment generating function of $\tilde{X}_{1}$ under different parameter values}
  \label{fig:1}
\end{figure}

Then, we report in Table \ref{tab:1} the first 6 (including the zeroth) moments of $\tilde{X}_{1}$, using Theorem \ref{th:6}.

\begin{table}[H]\scriptsize
  \centering
  \def\arraystretch{1.2}
  \caption{Moments of $\tilde{X}_{1}$ under different parameter values}
  \label{tab:1}
  \begin{tabular}{c|cccccc}
    \hline
    \multicolumn{7}{c}{$c\searrow0$} \\
    \hline
    $a;M_{\tilde{X}_{1}}(n);n$ & $0$ & $1$ & $2$ & $3$ & $4$ & $5$ \\
    \hline
    $\frac{1}{2}$ & $1$ & $\frac{1}{4}$ & $\frac{11}{48}$ & $\frac{25}{64}$ & $\frac{3839}{3840}$ & $\frac{3537}{1024}$ \\
    $1$ & $1$ & $\frac{1}{2}$ & $\frac{7}{12}$ & $\frac{9}{8}$ & $\frac{743}{240}$ & $\frac{1075}{96}$ \\
    $\frac{3}{2}$ & $1$ & $\frac{3}{4}$ & $\frac{17}{16}$ & $\frac{147}{64}$ & $\frac{8709}{1280}$ & $\frac{26499}{1024}$ \\
    $2$ & $1$ & $1$ & $\frac{5}{3}$ & $4$ & $\frac{191}{15}$ & $51$ \\
    \hline\hline
    \multicolumn{7}{c}{$c=\frac{1}{2}$} \\
    \hline
    $a;M_{\tilde{X}_{1}}(n);n$ & $0$ & $1$ & $2$ & $3$ & $4$ & $5$ \\
    \hline
    $\frac{1}{2}$ & $1$ & $\frac{\sqrt{\pi}}{4}$ & $\frac{3\pi+4\sqrt{\pi}}{48}$ & $\frac{\sqrt{\pi}(\pi+4\sqrt{\pi}+6)}{64}$ & $\frac{3\pi^{2}+24\pi^{3/2}+88\pi+144\sqrt{\pi}}{768}$ & $\frac{\sqrt{\pi}(3\pi^{2}+40\pi^{3/2}+260\pi+960\sqrt{\pi}+1680)}{3072}$ \\
    $1$ & $1$ & $\frac{\sqrt{\pi}}{2}$ & $\frac{3\pi+2\sqrt{\pi}}{12}$ & $\frac{\sqrt{\pi}(2\pi+4\sqrt{\pi}+3)}{16}$ & $\frac{3\pi^{2}+12\pi^{3/2}+22\pi+18\sqrt{\pi}}{48}$ & $\frac{\sqrt{\pi}(3\pi^{2}+20\pi^{3/2}+65\pi+120\sqrt{\pi}+105)}{96}$ \\
    $\frac{3}{2}$ & $1$ & $\frac{3\sqrt{\pi}}{4}$ & $\frac{9\pi+4\sqrt{\pi}}{16}$ & $\frac{9\sqrt{\pi}(3\pi+4\sqrt{\pi}+2)}{64}$ & $\frac{3(27\pi^{2}+72\pi^{3/2}+88\pi+48\sqrt{\pi})}{256}$ & $\frac{3\sqrt{\pi}(81\pi^{2}+360\pi^{3/2}+780\pi+960\sqrt{\pi}+560)}{1024}$ \\
    $2$ & $1$ & $\sqrt{\pi}$ & $\frac{3\pi+\sqrt{\pi}}{3}$ & $\frac{\sqrt{\pi}(8\pi+8\sqrt{\pi}+3)}{8}$ & $\frac{12\pi^{2}+24\pi^{3/2}+22\pi+9\sqrt{\pi}}{12}$ & $\frac{\sqrt{\pi}(48\pi^{2}+160\pi^{3/2}+260\pi+240\sqrt{\pi}+105)}{48}$ \\
    \hline
  \end{tabular}
\end{table}

Obviously, from (\ref{2.28}) we have the simple relation that $C_{\tilde{G}_{t}}(n)\propto a$, for any given $t$, $b$ and $n$. Also, the first cumulant and the first moment must coincide. It is also seen that specialized results are obtainable in the gamma and inverse Gaussian cases, while in general the moments will involve gamma functions.

Evaluating the probability density function and the cumulative distribution function requires numerical computation of the integral representations (\ref{2.14}) and (\ref{2.17}) with (\ref{2.18}). This can be efficiently carried out by using the numerical integration function of Mathematica$^{\circledR}$ by [Wolfram Research, Inc., 2015] \cite{W}, which uses by default the Gauss-Kronrod quadrature rule. Recall that the integrands in both (\ref{2.14}) and (\ref{2.17}) are continuous and bounded over the unit interval. Hence, in Figure \ref{fig:2} (in two pages) we plot the probability density function of $\tilde{X}_{1}$ using different shape and family parameters, where for ease of comparison the density function of $X_{1}$ is also included. As a reminder, we have explicitly
\begin{equation}\label{3.1}
  f_{X_{t}}(x|c\searrow0)=\frac{b^{at}}{\Gf(at)}x^{at-1}e^{-bx}\quad\text{and}\quad f_{X_{t}}\bigg(x\bigg|c=\frac{1}{2}\bigg)=\frac{at}{x^{3/2}}\exp\bigg(-\frac{(\sqrt{b}x-\sqrt{\pi}at)^{2}}{x}\bigg),
\end{equation}
for $x>0$, which are the $\mathrm{G}(at,b)$ and $\mathrm{IG}(at,b)$ density functions, respectively.

\begin{figure}[H]
  \centering
  \begin{minipage}[c]{0.49\linewidth}
  \centering
  \includegraphics[width=2.5in]{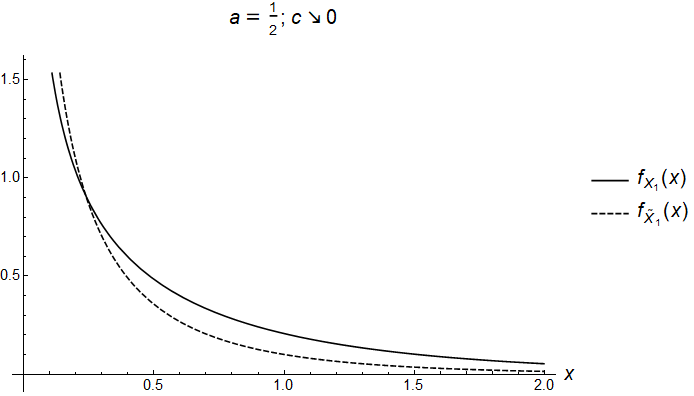}
  \includegraphics[width=2.5in]{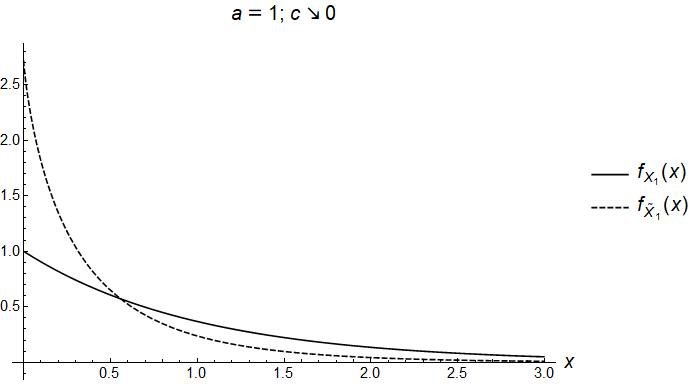}
  \includegraphics[width=2.5in]{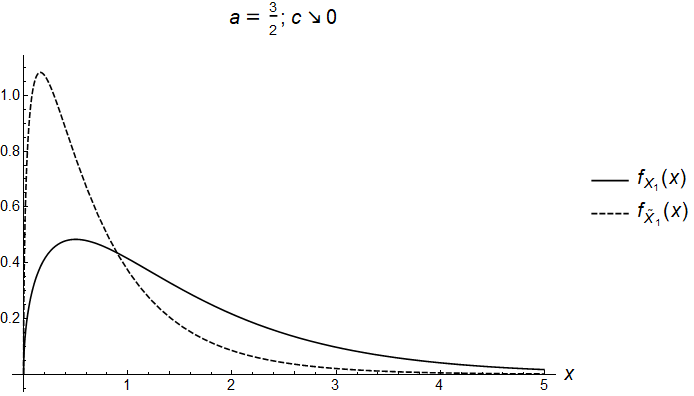}
  \includegraphics[width=2.5in]{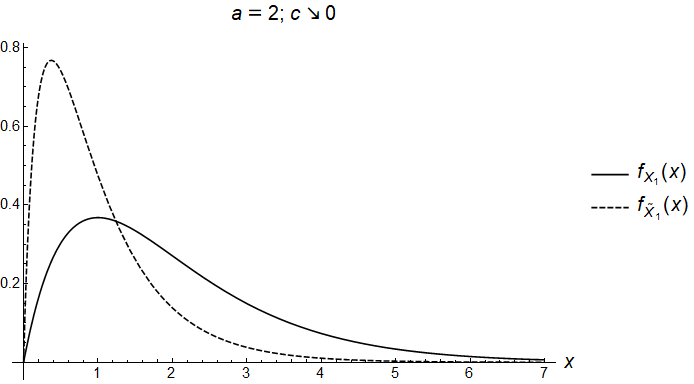}
  \end{minipage}
  \begin{minipage}[c]{0.49\linewidth}
  \centering
  \includegraphics[width=2.5in]{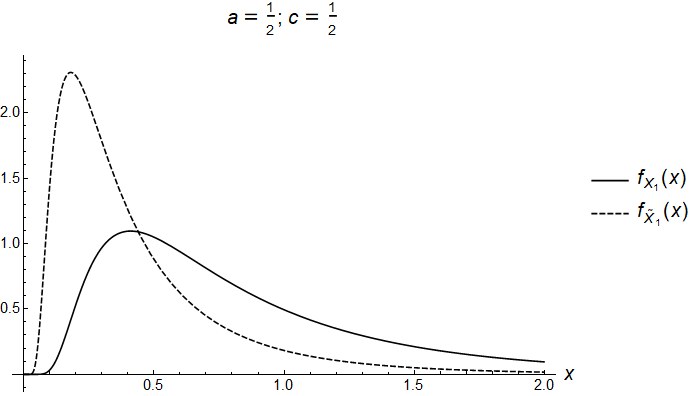}
  \includegraphics[width=2.5in]{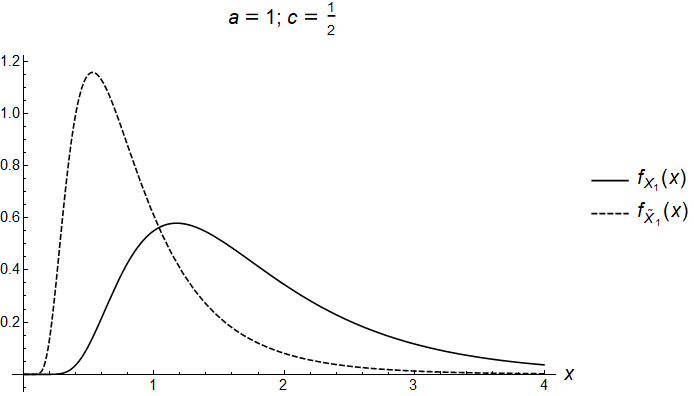}
  \includegraphics[width=2.5in]{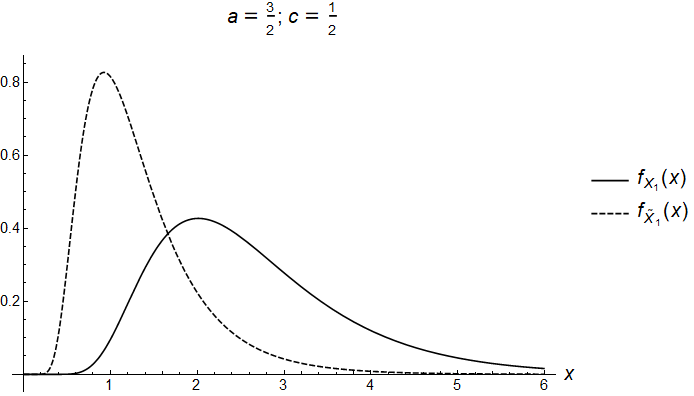}
  \includegraphics[width=2.5in]{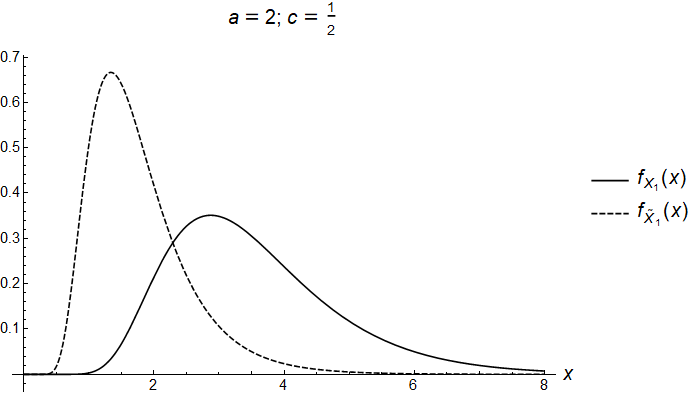}
  \end{minipage}
  \caption{Probability density function of $\tilde{X}_{1}$ under different parameter values}
  \label{fig:2}
\end{figure}

It is seen that by averaging the density function of $\tilde{X}_{1}$ becomes steeper relative to that of $X_{1}$, regardless of the value of $a$ or $c$. Also, as convinced by (\ref{2.35}) the $\mathrm{ATS}(a,b;c)$ has a more skewed and heavier right tail compared with the $\mathrm{TS}(a,b;c)$ distribution under the same parameter values. Another observation is that, in the gamma case, $\lim_{x\searrow0}\tilde{f}_{\tilde{X}_{1}}(x|c\searrow0)=\lim_{x\searrow0}\tilde{f}_{X_{1}}(x|c\searrow0)=\infty$ for $a<1$ and $\lim_{x\searrow0}\tilde{f}_{\tilde{X}_{1}}(x|c\searrow0)=\lim_{x\searrow0}\tilde{f}_{X_{1}}(x|c\searrow0)=0$ for $a>1$, while for $a=1$ these two limits exist and are strictly positive. This generally agrees with the discussion of Corollary \ref{cr:4}, which tells along with (\ref{3.1}) further that these limits equal $e$ and 1, respectively, with $b=1$. In the inverse Gaussian case, on the other hand, we have $\lim_{x\searrow0}\tilde{f}_{\tilde{X}_{1}}(x|c=1/2)=0$ uniformly in $a$.

The cumulative distribution function of $\tilde{X}_{1}$ is plotted in Figure \ref{fig:3} on the next page. \clearpage

\begin{figure}[H]
  \centering
  \begin{minipage}[c]{0.49\linewidth}
  \centering
  \includegraphics[width=2.5in]{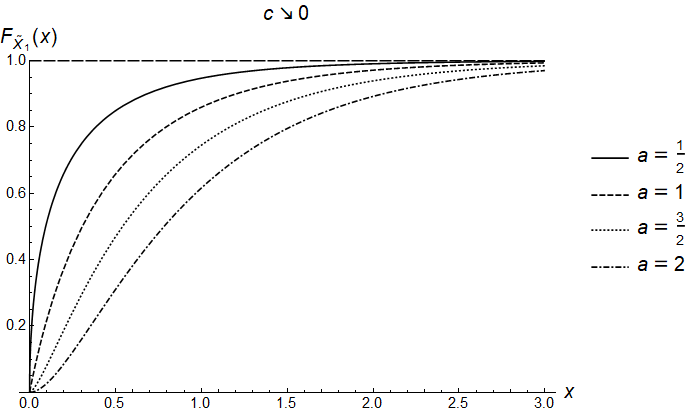}
  \end{minipage}
  \begin{minipage}[c]{0.49\linewidth}
  \centering
  \includegraphics[width=2.5in]{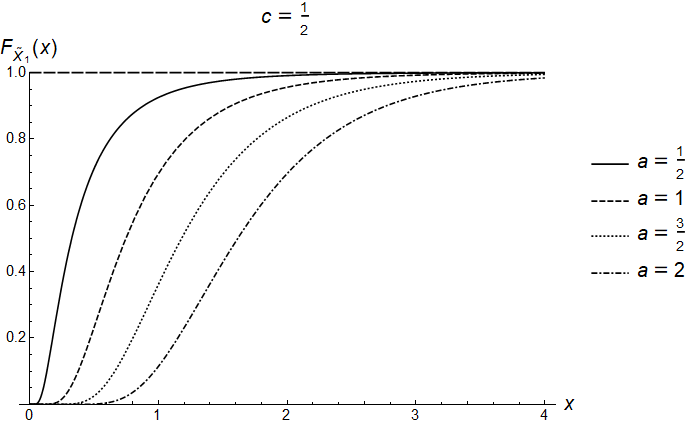}
  \end{minipage}
  \caption{Cumulative distribution function of $\tilde{X}_{1}$ under different parameter values}
  \label{fig:3}
\end{figure}

\vspace{0.2in}

\section{Extensions and applications}\label{sec:4}

We discuss two natural functional extensions of the running average process $\tilde{X}$ and its equi-distributed average-tempered stable subordinator $\Lambda$. As aforementioned in Section \ref{sec:1}, we will focus on applications in the context of structural degradation and financial derivatives pricing, which in general can be efficiently implemented thanks to the distributional formulas proposed in Section \ref{sec:2}.

\vspace{0.2in}

\subsection{Degradation modeling}

The random occurrence of degradation of a structural component, as commonly modeled by the gamma process $G$, is naturally a stochastic process with monotone sample paths. Since the tempered stable subordinator $X$ generalizes $G$ by having more complicated path behaviors while still remaining nonnegative and nondecreasing, it is natural to use $X$ as a model for the degradation process. In contrast to having purely discontinuous sample paths as $X$ does, another popular degradation model is simply the drift process $(At)$, which is continuous and also features a constant instantaneous rate of degradation $A>0$.

We recall from Section \ref{sec:1} that the running average process $\tilde{X}$ of $X$ is a nondecreasing and continuous process starting from 0 and hence allows for degradation not only in a continuous manner but at a time-dependent rate as well. In particular, in this case the instantaneous rate of degradation $A\equiv(A_{t})$ varies with the passage of time and is such that
\begin{equation*}
  \PP\bigg\{\int^{t}_{0}A_{s}\dd s=\tilde{X}_{t},\;\forall t\geq0\bigg\}=1,
\end{equation*}
which is precisely given by
\begin{equation*}
  A_{t}=\frac{X_{t}-\tilde{X}_{t}}{t},\quad t>0.
\end{equation*}
At $t=0$, with $\E e^{-uX_{t}/t}=1+O(t)$ for $u\in\mathds{C}\setminus(-\infty,-b]$, an application of L\'{e}vy's continuity theorem implies $\lim_{t\searrow0}X_{t}/t=0$, $\PP$-a.s. The fact that $A_{0}$ is $\PP$-\text{a.s.} 0 then results from dominated convergence applied to the difference $X_{t}-\tilde{X}_{t}$, for any $t>0$. On the other hand, applying It\^{o}'s formula yields the dynamics
\begin{equation*}
  \dd A_{t}=\frac{\dd X_{t}-2A_{t}\dd t}{t},\quad A_{0}=0,
\end{equation*}
which further shows that under $\tilde{X}$, the instantaneous degradation rate is a stochastic process which exhibits mean reversion and decreasing variance over time\footnote{We notice that stabilizing trends are a common trait of many degradation phenomena. See, e.g., [Wang et al, 2015, $\S$4.1] \cite{WXM} or the application to be presented in this section.}.

Let $l>0$ be the initial observed condition of the structural component and consider the stopping time $\tau_{l}:=\inf\{t>0:\tilde{X}_{t}\geq l\}$. Then the evolution of the component condition can be modeled as the stopped flipped process $D\equiv(D_{t}):=(l-\tilde{X}_{t\wedge\tau_{l}})$. Since the path continuity of $\tilde{X}$ renders $\tau_{l}$ a hitting time which rules out the possibility of overshoot, we can write equivalently
\begin{equation*}
  D_{t}=(l-\tilde{X}_{t})^{+},\quad t\geq0,
\end{equation*}
where $(\cdot)^{+}$ denotes the positive part and for which $D_{\tau_{l}}=0$, $\PP$-a.s.

Using Theorem \ref{th:2} and Theorem \ref{th:3}, computation of the expected condition at some fixed time $T>0$ of observation is straightforward, in terms of
\begin{align*}
  \E D_{T}&=l\PP\{\tilde{X}_{T}\leq l\}-\int^{l}_{0}xf_{\tilde{X}_{T}}(x)\dd x\\
  &=lF_{\tilde{X}_{T}}(l)-\frac{e^{-atb^{c}\Gf(-c)}}{\pi b}\int^{1}_{0}\frac{\dd y}{y}\;\exp\Bigg(-atb^{c}\Gf(-c-1)\Bigg(y+\cos(\pi c)(1-y)\bigg(\frac{1}{y}-1\bigg)^{c}\Bigg)\Bigg)\\
  &\qquad\times\sin\Bigg(atb^{c}\sin(\pi c)\Gf(-c-1)(1-y)\bigg(\frac{1}{y}-1\bigg)^{c}\Bigg)(y-(bl+y)e^{-bl/y}).
\end{align*}
The probability that the condition level at time $T$ will be higher than some fixed lower alert level or barrier $\underline{D}\in(0,l)$ is then
\begin{equation}\label{4.1.1}
  \PP\{D_{T}>\underline{D}\}=\int^{\infty}_{0}\mathds{1}_{\{(l-x)^{+}>\underline{D}\}}f_{\tilde{X}_{T}}(x)\dd x=F_{\tilde{X}_{T}}(l-\underline{D})\equiv\PP\{\tilde{X}_{T}\leq l-\underline{D}\},
\end{equation}
which is nothing but the survival function of $D_{T}$. In other words, the event $\{D_{T}\leq\underline{D}\}$ is deemed to be the failure of the structural component at time $T$. The lifetime of the structural component is hence identified as the first passage time $\tau_{\underline{D}}:=\inf\{t>0:\tilde{D}_{t}<\underline{D}\}$, and by path monotonicity has the distribution $\PP\{\tau_{\underline{D}}>T\}=F_{\tilde{X}_{T}}(l-\underline{D})$, for $T>0$, as well, and the density function of $\tau_{\underline{D}}$ can be obtained via differentiating $1-F_{\tilde{X}_{T}}(l-\underline{D})$ in $T$. However, it is not a trivial task to extrapolate a convenient formula for the expected lifetime because of noninterchangeable integrals, as in the simple gamma model. Instead, one may consider the median lifetime, which is given by the solution $T^{\rm m}>0$ of $F_{\tilde{X}_{T^{\rm m}}}(l-\underline{D})=1/2$ and is also always existent.

Notice that $\tilde{X}$ cannot be a Markov process by its construction (\ref{1.4}), and hence performing parameter estimation based on the transition density will be cumbersome. Despite this, the non-Markovian property is deemed benign by incorporating memory into the degradation process, i.e., the past condition can influence future degradation behavior. In this case, the data will need to be modified in order to be accessible for estimation. Suppose we observe the degradation levels of a certain structural component at times $0=t_{0}<t_{1}<\cdots<t_{M-1}<t_{M}=T$, not necessarily equally spaced, before a predetermined date $T>0$, denoted by $\{\check{\tilde{X}}_{t_{m}}\}_{m\in\mathds{N}\cap[1,M]}$ with $\check{\tilde{X}}_{0}=0$. Then, under the running average model $\tilde{X}$, one choice is to perform the approximate transformation
\begin{equation}\label{4.1.2}
  \check{X}_{t_{m}}:=\frac{t_{m}\check{\tilde{X}}_{t_{m}}-t_{m-1}\check{\tilde{X}}_{t_{m-1}}}{t_{m}-t_{m-1}}, \quad\text{and}\quad\check{\xi}_{m}:=\check{X}_{t_{m}}-\check{X}_{t_{m-1}},
\end{equation}
with $\check{X}_{0}=0$, and where $\check{\xi}_{m}$'s, for $m\in\mathds{N}\cap[2,M]$, can be viewed as independent random variables, each having a corresponding $\mathrm{TS}(a(t_{m}-t_{m-1}),b;c)$ distribution, with $\check{\xi}_{1}=\check{\tilde{X}}_{t_{1}}$. In other words, the running average model postulates that the scaled second-order difference sequence of the data consists of independent tempered stable-distributed elements. This transformation enables us to employ various estimation methods intended for tempered stable distributions.

It is worth mentioning that the probability density function of $X_{t}$, like that of $\tilde{X}_{t}$ provided in (\ref{2.14}), does not have a general closed form, except in the gamma and inverse Gaussian cases, as in (\ref{3.1}). For this reason, parameter estimation is oftentimes deemed challenging with respect to the family parameter $c$, though due to the closed-form Laplace transform in (\ref{1.1}) both the empirical characteristic function estimation method discussed in [Yu, 2004] \cite{Y2} and the more oriented quantile-based inference method proposed in [Fallahgoul et al, 2019] \cite{FVF} can still be employed. In the following application, for convenience we only consider the gamma and inverse Gaussian cases, with $c\searrow0$ and $c=1/2$, respectively, for which (\ref{3.1}) permits direct implementation of maximum likelihood estimation with the log-likelihood function
\begin{equation*}
  L(a,b|c;\check{\xi}_{2},\dots,\check{\xi}_{M})=\sum^{M}_{m=2}\log f_{X_{t_{m}-t_{m-1}}}(\check{\xi}_{m}|a,b;c),
\end{equation*}
which is to be maximized over $a,b>0$ in order to determine the corresponding estimates $\hat{a}$ and $\hat{b}$. The Akaike information criterion (AIC) value, which measures the relative amount of lost information and therefore the quality of the model, can be computed as $4-2\max_{a,b>0}L(a,b|c;\check{\xi}_{2},\dots,\check{\xi}_{M})$.

We use the degradation data for carbon-film resistors shown in [Meeker and Escobar, 1998, Table C.3] \cite{ME}, which consist of a total of 29 resistors observed at 3 different temperatures - 83$^\circ$C, 133$^\circ$C, and 173$^\circ$C, and inspected at times $t_{0}=0$, $t_{1}=0.0452$, $t_{2}=0.103$, $t_{3}=0.4341$, and $t_{4}=0.8084$ in the unit of $10^{4}$ hours. The degradation is measured as the resistance level in the unit of ohms. We assume that the maximal degradation level permitted is universally set at 3.8 ohms.

Figure \ref{fig:4} on the next page shows the observed data in three groups. It is clear that the degradation paths of each resistor stabilize over time, with steeper trends in the beginning, regardless of the temperature, while under a higher temperature they tend to increase on a larger scale. In this respect, we can expect the average-tempered stable process, which captures decreasing variance, to serve as a good model.

\begin{figure}[H]
  \centering
  \begin{minipage}{0.32\linewidth}
  \centering
  \includegraphics[width=1.9in]{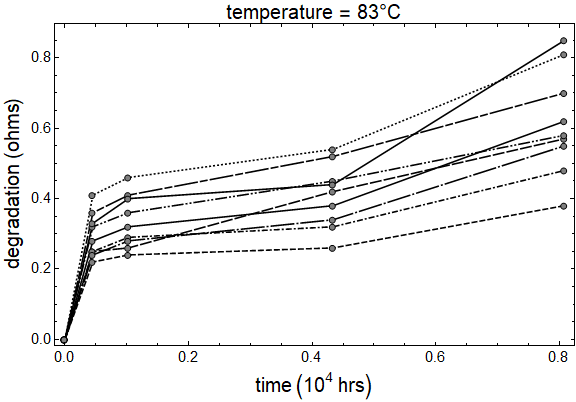}
  \end{minipage}
  \begin{minipage}{0.32\linewidth}
  \centering
  \includegraphics[width=1.9in]{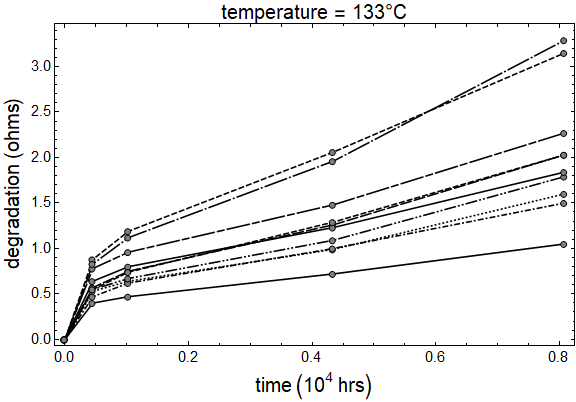}
  \end{minipage}
  \begin{minipage}{0.32\linewidth}
  \centering
  \includegraphics[width=1.9in]{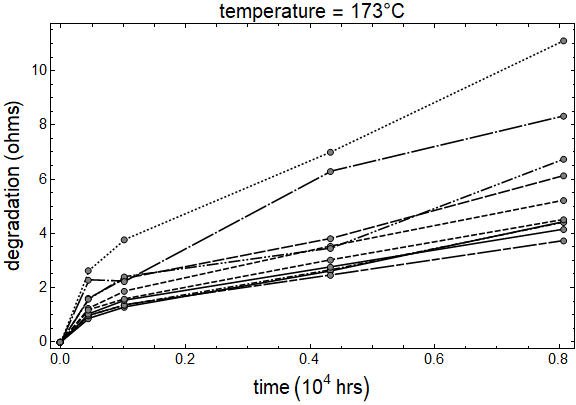}
  \end{minipage}
  \caption{Degradation data for carbon-film resistors}
  \label{fig:4}
\end{figure}

Using the transformation (\ref{4.1.2}) we then perform maximum likelihood estimation on each sample resistor using the gamma model, the average-gamma model, as well as the average-inverse Gaussian model, in proper order, taking only $c\searrow0$ or $c=1/2$. The parameter estimates together with the AIC values are reported in Table \ref{tab:2}. We note that in the 9th and 27th samples, some observed degradation values are negative, which jeopardize the likelihood function and applicability of the models, and the corresponding results are marked as ``NA''.

\begin{table}[H]\scriptsize
  \centering
  \caption{Maximum likelihood estimation results (rounded to 4 decimal places)}
  \label{tab:2}
  \begin{tabular}{c|c|c|c|c|c|c|c|c|c|c}
    \hline
    \multicolumn{2}{c|}{condition} & \multicolumn{3}{c|}{gamma} & \multicolumn{3}{c|}{average-gamma} & \multicolumn{3}{c}{average-inverse Gaussian} \\ \hline
    temperature & index & $\hat{a}$ & $\hat{b}$ & AIC & $\hat{a}$ & $\hat{b}$ & AIC & $\hat{a}$ & $\hat{b}$ & AIC \\ \hline
    \multirow{9}{*}{83$^\circ$C} & 1 & 5.4392 & 7.0920 & 1.2904 & 4.9303 & 6.0852 & 1.1573 & 0.5301 & 1.3451 & 1.5320 \\
     & 2 & 3.5653 & 7.5846 & $-1.7329$ & 3.1469 & 8.0278 & $-2.7409$ & 0.2296 & 1.0774 & $-2.0652$ \\
     & 3 & 4.9885 & 4.9786 & 3.7959 & 5.7035 & 6.1039 & 1.9077 & 0.6477 & 1.5095 & 2.3166 \\
     & 4 & 4.0520 & 6.8242 & 0.5441 & 2.4543 & 4.5074 & $-0.2534$ & 0.1941 & 0.3991 & 1.0559 \\
     & 5 & 7.3695 & 10.4518 & $-2.3447$ & 94.0743 & 145.3494 & $-9.9406$ & 2.7731 & 57.6728 & $-10.0055$ \\
     & 6 & 5.0906 & 7.0952 & 1.5135 & 15.3063 & 28.4388 & $-3.1779$ & 1.0332 & 11.5762 & $-2.7026$ \\
     & 7 & 5.6509 & 6.5260 & 2.6638 & 13.2296 & 18.3995 & $-1.1643$ & 1.0739 & 7.0064 & $-0.6723$ \\
     & 8 & 5.9010 & 8.6735 & 0.2632 & 5.3647 & 7.3965 & 0.5454 & 0.5412 & 1.7490 & 0.9919 \\
     & 9 & 4.2317 & 4.0246 & 4.4242 & NA & NA & NA & NA & NA & NA \\ \hline
    \multirow{10}{*}{133$^\circ$C} & 10 & 9.6038 & 7.3940 & 3.5696 & 29.8305 & 22.0453 & $-0.1489$ & 2.3710 & 9.6456 & 0.1960 \\
     & 11 & 15.0100 & 3.8521 & 10.7812 & 23.8076 & 5.1413 & 8.3236 & 3.9549 & 2.2915 & 8.7277 \\
     & 12 & 11.2990 & 5.7088 & 6.1122 & 16.6526 & 7.1502 & 4.6030 & 2.1930 & 2.7855 & 4.8861 \\
     & 13 & 12.2686 & 6.6120 & 5.8478 & 19.5432 & 9.2649 & 4.2671 & 2.3817 & 4.0050 & 4.6998 \\
     & 14 & 14.7550 & 5.8758 & 7.1451 & 18.7597 & 6.0847 & 6.3170 & 2.7384 & 2.4784 & 6.7307 \\
     & 15 & 12.9357 & 5.8420 & 6.3269 & 19.0215 & 7.0752 & 4.9993 & 2.5547 & 2.8366 & 5.2409 \\
     & 16 & 10.8622 & 3.8683 & 9.3775 & 18.9504 & 6.0107 & 6.3846 & 2.7820 & 2.4460 & 6.7885 \\
     & 17 & 17.3109 & 4.2535 & 10.1754 & 17.9011 & 3.4134 & 9.5417 & 3.4573 & 1.3653 & 9.9315 \\
     & 18 & 10.6035 & 4.6586 & 7.9710 & 19.2656 & 7.7084 & 5.1419 & 2.5294 & 3.2176 & 5.5790 \\
     & 19 & 15.8862 & 6.3263 & 6.8971 & 22.5786 & 7.3697 & 5.8624 & 3.0776 & 3.1701 & 6.2776 \\ \hline
    \multirow{10}{*}{173$^\circ$C} & 20 & 27.9058 & 5.0809 & 10.3513 & 24.4578 & 3.2859 & 10.7986 & 4.9838 & 1.4084 & 11.1860 \\
     & 21 & 17.8103 & 2.7529 & 14.2527 & 25.8732 & 3.3243 & 11.5411 & 5.6924 & 1.6805 & 11.7795 \\
     & 22 & 19.4770 & 1.4159 & 19.6218 & 23.6622 & 1.3633 & 16.2354 & 7.6026 & 0.6028 & 16.6437 \\
     & 23 & 23.1215 & 4.2288 & 11.1313 & 27.1266 & 3.7771 & 10.1776 & 5.1773 & 1.6326 & 10.5357 \\
     & 24 & 14.5942 & 1.9215 & 16.3966 & 12.7992 & 1.3547 & 14.5360 & 3.8504 & 0.5218 & 15.0424 \\
     & 25 & 11.1269 & 1.3326 & 17.9718 & 5.1440 & 0.4375 & 17.7285 & 2.1111 & 0.1012 & 18.3516 \\
     & 26 & 16.4732 & 3.5607 & 11.7835 & 26.3783 & 4.7561 & 9.2008 & 4.6630 & 2.2207 & 9.5611 \\
     & 27 & NA & NA & NA & NA & NA & NA & NA & NA & NA \\
     & 28 & 16.6790 & 3.2412 & 12.7439 & 21.4549 & 3.4603 & 10.6977 & 4.4713 & 1.6338 & 11.0347 \\
     & 29 & 17.1961 & 3.0755 & 12.8759 & 47.2922 & 7.1358 & 8.3138 & 7.0149 & 3.5197 & 8.5934 \\
    \hline
  \end{tabular}
\end{table}

From comparing the AIC values across each sample, we conclude that the models based on the average processes perform better than the traditional gamma model in general. More specifically, the performance of the average-gamma model is superior to that of the average-inverse Gaussian model, except for the 5th sample, and only for the 1st, 20th, and 25th samples does the average-inverse Gaussian model perform worse than the gamma model. These comments highlight to some extent the conspicuousness of stabilization in the degradation of carbon-film resistors, which certainly cannot be taken into account by any L\'{e}vy process. Of course, estimation of the average-gamma and average-inverse Gaussian model parameters will be more robust if a larger data set with more frequent observations is available, and the performance can also be improved if the family parameter $c$ is incorporated, in one way or another, into the estimation procedure.

Next we compute the survival probability of the sample resistors subject to the barrier of 3.8 ohms $10^{4}$ hours from the initial observation, using (\ref{4.1.1}) and the estimated parameters. Since $\check{\xi}_{1}=\check{\tilde{X}}_{t_{1}}$, under the average-gamma and average-inverse Gaussian models we should use $T=1-t_{1}=0.9548$ and $l-\underline{D}=3.8-\check{\tilde{X}}_{t_{1}}$ for each sample. Table \ref{tab:3} shows the results.

\begin{table}[H]\scriptsize
  \centering
  \caption{Survival probabilities (rounded to 6 decimal places)}
  \label{tab:3}
  \begin{tabular}{c|c|c|c|c}
    \hline
    \multicolumn{2}{c|}{condition} & \multicolumn{3}{c}{survival probability} \\ \hline
    temperature & index & gamma & average-gamma & average-inverse Gaussian \\ \hline
    \multirow{9}{*}{83$^\circ$C} & 1 & 1.000000 & 1.000000 & 0.999853 \\
     & 2 & 1.000000 & 1.000000 & 0.999892 \\
     & 3 & 1.000000 & 1.000000 & 0.999851 \\
     & 4 & 0.999960 & 1.000000 & 0.996929 \\
     & 5 & 1.000000 & 1.000000 & 1.000000 \\
     & 6 & 1.000000 & 1.000000 & 1.000000 \\
     & 7 & 0.999999 & 1.000000 & 1.000000 \\
     & 8 & 1.000000 & 1.000000 & 0.999975 \\
     & 9 & 0.999764 & NA & NA \\ \hline
    \multirow{10}{*}{133$^\circ$C} & 10 & 0.999983 & 1.000000 & 1.000000 \\
     & 11 & 0.495816 & 0.900246 & 0.890481 \\
     & 12 & 0.994656 & 0.999995 & 0.999754 \\
     & 13 & 0.998355 & 1.000000 & 0.999992 \\
     & 14 & 0.963505 & 0.999527 & 0.997138 \\
     & 15 & 0.986878 & 0.999968 & 0.999390 \\
     & 16 & 0.874649 & 0.998246 & 0.993390 \\
     & 17 & 0.420170 & 0.767436 & 0.765238 \\
     & 18 & 0.972725 & 0.999989 & 0.999731 \\
     & 19 & 0.968141 & 0.999878 & 0.999054 \\ \hline
    \multirow{10}{*}{173$^\circ$C} & 20 & 0.038494 & 0.241870 & 0.264384 \\
     & 21 & 0.023811 & 0.071724 & 0.058346 \\
     & 22 & 0.000002 & 0.000002 & 0.000000 \\
     & 23 & 0.057259 & 0.224485 & 0.244850 \\
     & 24 & 0.011282 & 0.028665 & 0.028403 \\
     & 25 & 0.013324 & 0.077709 & 0.143482 \\
     & 26 & 0.245161 & 0.641475 & 0.654520 \\
     & 27 & NA & NA & NA \\
     & 28 & 0.136732 & 0.429712 & 0.448145 \\
     & 29 & 0.077529 & 0.153176 & 0.149100 \\
    \hline
  \end{tabular}
\end{table}

We see that, relative to the gamma model, the two models based on average processes are prone to making more conservative predictions about survival, and more so as temperature rises. In general, it is understood that under a higher temperature the resistors undergo more violent degradation, leading to vastly dropping survival probabilities.

Noted from Table \ref{tab:3} that the 17th sample resistor features the closest probabilities to 0.5, using the corresponding parameter values in Table \ref{tab:2}, we compute the median lifetime of the 17th sample, which are approximately 0.952899, 1.205356, and 1.225898 in $10^{4}$ hours, under the gamma model, the average-gamma model, and the inverse-Gaussian model, in proper order, all of which are close to $T=1$, whereas implications from the average process models are more optimistic relative to the gamma model.

One can also apply Euler's discretization scheme to simulate the degradation paths. Let $M_{i}$ be the total number of discretization points and the simulation can be carried out on the unit time interval as follows,
\begin{align*}
  &\hat{X}_{0}=0\rightsquigarrow\hat{X}_{(i+1)(1/M_{i})}=\hat{X}_{i(1/M_{i})}+\xi_{i+1},\\
  &\hat{X}'_{0}=0\rightsquigarrow\hat{X}'_{(i+1)(1/M_{i})}=\hat{X}'_{i(1/M_{i})}+X_{(i+1)(1/M_{i})},\\
  &\hat{\tilde{X}}_{i(1/M_{i})}=\frac{\hat{X}'_{i(1/M_{i})}}{i(1/M_{i})},\quad i\in\mathds{N}\cap[0,M_{i}].
\end{align*}
where $\{\xi_{i}\}$ for $i\geq1$ is a sequence of \text{i.i.d.} $\mathrm{TS}(a(1/M_{i}),b;c)$ random variables. The trajectories of $\{\hat{X}_{i}\}$ and $\{\hat{\tilde{X}}_{i}\}$ are approximations of the sample paths of $X$ and $\tilde{X}$, respectively. Figure \ref{fig:5} focuses on the 17th sample and realizes 100 degradation paths based on $M_{i}=2000$ under each of the three models applied, where for comparison both the barrier and the observed degradation path from the data are included. For the two average process models the time scale is also curtailed by the initial $t_{1}=0.0452$ in consideration of $\check{\xi}_{1}>0$.

\begin{figure}[H]
  \centering
  \begin{minipage}{0.32\linewidth}
  \centering
  \includegraphics[width=1.9in]{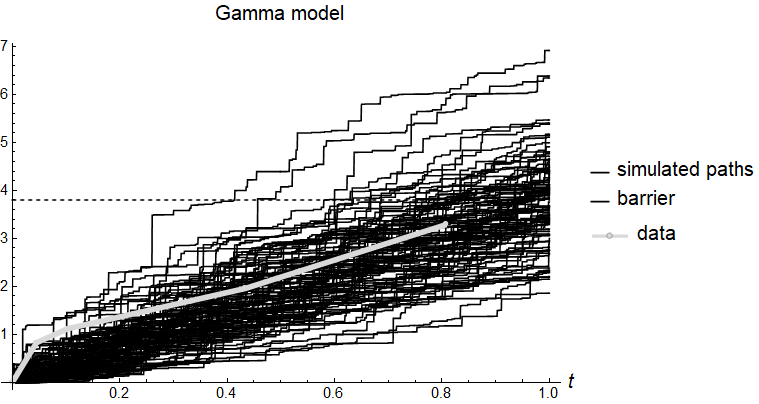}
  \end{minipage}
  \begin{minipage}{0.32\linewidth}
  \centering
  \includegraphics[width=1.9in]{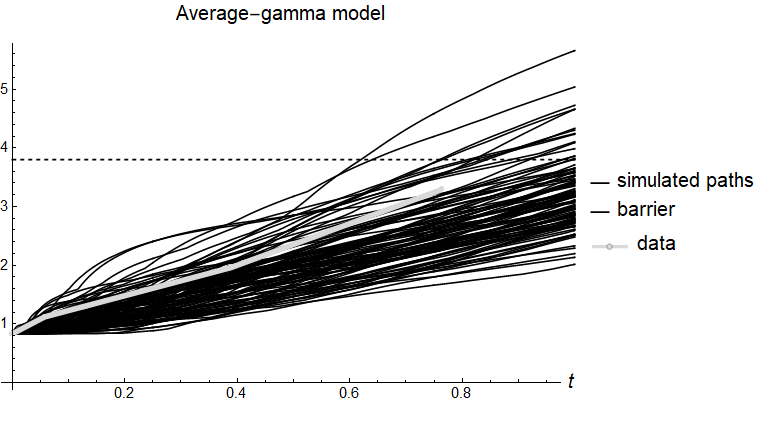}
  \end{minipage}
  \begin{minipage}{0.32\linewidth}
  \centering
  \includegraphics[width=1.9in]{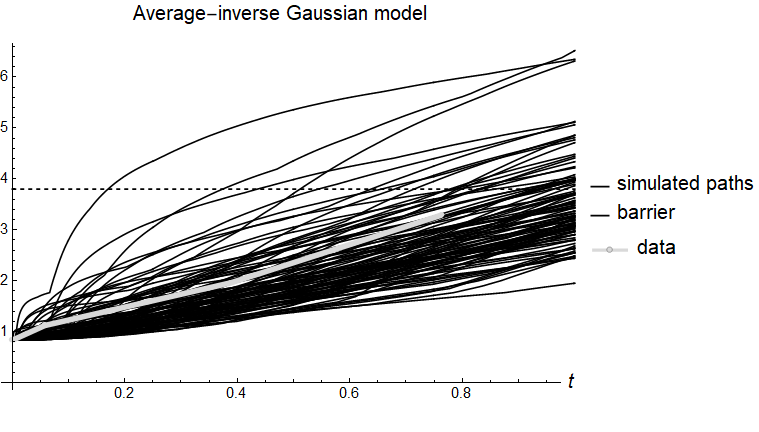}
  \end{minipage}
  \caption{Simulated degradation paths for 17th sample}
  \label{fig:5}
\end{figure}

Observably, the realized paths under the gamma process cross the barrier without touching it. In fact, since $X$ is a subordinator with infinite activity it hits any horizontal barrier with probability 0. However, the realized continuous paths under the average processes, as described by $\tilde{X}$, cross the barrier by touching it indeed. Furthermore, apart from path continuity, the average processes are more suitable for the actual degradation behavior of the resistors due to their stabilizing trends in nature.

\vspace{0.2in}

\subsection{Gaussian mixture and financial derivatives pricing}

Recall that the average-tempered stable subordinator $\Lambda$ induced by the distribution of $\tilde{X}_{t}$ for $t\geq0$ has nonnegative and nondecreasing sample paths, and therefore can be used as a random time change. In connection with this let $W\equiv(W_{t})$ be a standard Brownian motion which is independent from $\Lambda$ and define the corresponding time-changed drifted Brownian motion
\begin{equation*}
  H_{t}:=\kappa t+\mu\Lambda_{t}+\sigma W_{\Lambda_{t}},\quad t\geq0,
\end{equation*}
where $\kappa\in\mathds{R}$ is the drift parameter, $\mu\in\mathds{R}$ the Brownian location parameter and $\sigma>0$ the Brownian scale parameter. Since the sample paths of $\Lambda$ are discontinuous $\PP$-a.s., $H$ is automatically a purely discontinuous L\'{e}vy process, which is also of infinite activity. Since $\mu t+\sigma W_{t}$ is normally distributed with mean $\mu t$ and variance $\sigma^{2}t$ for any $t>0$, by the law of total expectation, using Theorem \ref{th:1} it can be shown that $H_{t}$ has the Laplace transform
\begin{equation}\label{4.2.1}
  \bar{f}_{H_{t}}(u):=\E e^{-uH_{t}}=e^{-\kappa tu}\bar{f}_{\tilde{X}_{t}}\bigg(\mu u-\frac{1}{2}\sigma^{2}u^{2}\bigg),
\end{equation}
which is well-defined for
\begin{equation*}
  u\in\mathds{C}\setminus\bigg(\bigg(-\infty,\frac{\mu-\sqrt{\mu^{2}+2b\sigma^{2}}}{\sigma^{2}}\bigg]\cup \bigg[\frac{\mu+\sqrt{\mu^{2}+2b\sigma^{2}}}{\sigma^{2}},\infty\bigg)\bigg).
\end{equation*}
This is substantially identical to the construction of the well-known variance gamma process as a gamma time-changed Brownian motion with drift in [Madan and Seneta, 1990] \cite{MS}. In light of the relation (\ref{4.2.1}) together with Theorem \ref{th:1}, it is easily justifiable that $H$ has the Blumenthal-Getoor index $\mathbf{B}(H)=2\mathbf{B}(\Lambda)=2c\in(0,2)$, which implies that $H$ has sample paths of finite variation if and only if $c\in(0,1/2)$, as in the case of the average-gamma subordinator - and not the average-inverse Gaussian subordinator. Further, by the Bayes theorem and the time-shifting property of Laplace transforms we employ Theorem \ref{th:2} to obtain the density function of $H_{t}$ for any fixed $t>0$,
\begin{align}\label{4.2.2}
  f_{H_{t}}(x)&:=\frac{\PP\{H_{t}\in\dd x\}}{\dd x}\\
  &=\int^{\infty}_{0}\frac{e^{-(x-\kappa t-\mu w)^{2}/(2\sigma^{2}w)}}{\sqrt{2\pi\sigma^{2}w}}f_{\tilde{X}_{t}}(w)\dd w \nonumber\\
  &=\frac{be^{-atb^{c}\Gf(-c)}}{\pi}\int^{1}_{0}\frac{\dd y}{y^{3/2}}\;\exp\Bigg(-atb^{c}\Gf(-c-1)\Bigg(y+\cos(\pi c)(1-y)\bigg(\frac{1}{y}-1\bigg)^{c}\Bigg)\Bigg) \nonumber\\
  &\qquad\times\sin\Bigg(atb^{c}\sin(\pi c)\Gf(-c-1)(1-y)\bigg(\frac{1}{y}-1\bigg)^{c}\Bigg)\frac{\psi(x-\kappa t,y)}{\sqrt{\mu^{2}y+2b\sigma^{2}}},\quad x\in\mathds{R}, \nonumber
\end{align}
in which
\begin{equation*}
  \psi(x,y)=\exp\frac{\mu x-|x|\sqrt{\mu^{2}+2b\sigma^{2}/y}}{\sigma^{2}},\quad(x,y)\in\mathds{R}\times(0,1).
\end{equation*}

As a L\'{e}vy process, $H$ automatically constitutes a model with jumps for risky financial asset prices. For example, consider a stock whose risk-neutralized\footnote{For a resume on the risk-neutral pricing theory of financial assets one may refer to [Schoutens, 2003] \cite{S1} as well as [Lyasoff, 2017] \cite{L2}.} log price evolves according to
\begin{equation*}
  \log S_{t}=\log\frac{S_{0}}{\bar{f}_{H_{t}}(-1)}+H_{t},\quad t\geq0,
\end{equation*}
with $e^{H}$ assumed to be integrable and the observed initial stock price $S_{0}$. Equivalently, the stock price process is given by the ordinary exponential $S=S_{0}e^{H}$. Then, the price of a European-style call option written on this stock with strike price $K$ and maturity $T>0$ can be computed as
\begin{equation}\label{4.2.3}
  \Pi^{\text{call}}_{0}=\E\big(e^{-rT}(S_{T}-K)^{+}\big)=S_{0}e^{-qT}\breve{P}-Ke^{-rT}P^{\ast},
\end{equation}
where the in-the-money probabilities are given by
\begin{equation*}
  P^{\ast}=\int^{\infty}_{\log(K/S_{0})-\kappa T}f_{H_{T}}(z)\dd z\quad\text{and}\quad
  \breve{P}=\int^{\infty}_{\log(K/S_{0})-\kappa T}zf_{H_{T}}(z)\dd z,
\end{equation*}
and $f_{H_{T}}$ is the probability density function of $H_{T}$ as given in (\ref{4.2.2}), which is already risk-neutral. According to [Bakshi and Madan, 2000] \cite{BM}, the above probabilities can be expressed in terms of numerical Fourier inverses based on (\ref{4.2.1}), i.e.,
\begin{align*}
  P^{\ast}&=\frac{1}{2}+\frac{1}{\pi}\int^{\infty}_{0}\Re\frac{(S_{0}\bar{f}_{H_{T}}(-1)/K)^{\ii u}\bar{f}_{H_{T}}(-\ii u)}{\ii u}\dd u,\\
  \breve{P}&=\frac{1}{2}+\frac{1}{\pi}\int^{\infty}_{0}\Re\frac{(S_{0}\bar{f}_{H_{T}}(-1)/K)^{\ii u}\bar{f}_{H_{T}}(-\ii u-1)}{\ii u\bar{f}_{H_{T}}(-1)}\dd u.
\end{align*}
Alternatively, Fubini's theorem applied to (\ref{4.2.2}) implies that $P^{\ast}$ and $\breve{P}$ can also be viewed as the probability density function $f_{H_{T}}$ evaluated at $x=\log(K/S_{0})-\kappa T$ with the function $\psi$ replaced by $\Psi^{\ast}(x,y):=\int^{\infty}_{x}\psi(z,y)\dd z$ and $\breve{\Psi}(x,y):=\int^{\infty}_{x}z\psi(z,y)\dd z$, respectively, for $(x,y)\in\mathds{R}\times(0,1)$, both of which can be evaluated explicitly.

Of course, according to the put-call parity, the price of a similar put option is
\begin{equation}\label{4.2.4}
  \Pi^{\text{put}}_{0}=\E\big(e^{-rT}(K-S_{T})^{+}\big)=\Pi^{\text{call}}_{0}+Ke^{-rT}-S_{0}e^{-qT}.
\end{equation}

Now we perform a simple calibration study on European-style Bitcoin option prices, whose return distribution is commonly known to have very large skewness and kurtosis (see, e.g., [Troster et al, 2019] \cite{TTSM}). In concrete, our data set consists of 40 Bitcoin option prices (denoted $\check{\Pi}^{\text{call}}$'s) quoted at the end of July 11, 2020 (data source: [Deribit, 2020] \cite{D}), with four different maturities $T=19,47,166,257$ days and 10 quotes per maturity. Strike prices range from \$3000 to \$32000. On that date, the Bitcoin index closed at $S_{0}=\$9232.98$.

We notice that the family parameter $c$ cannot be stably calibrated which but increases uncertainty for the values of other parameters and hence focus on four special cases $c\searrow0$ (gamma), $c=0.25$, $c=0.5$ (inverse Gaussian) and $c=0.75$. In other words, only four parameters will be calibrated on the yearly basis: $a>0$, $b>0$, $\mu\in\mathds{R}$ and $\sigma>0$; our optimization program targets minimizing the average relative pricing error (ARPE) so that the optimal parameters are given by
\begin{equation*}
  \{\hat{a},\hat{b},\hat{\mu},\hat{\sigma}\}=\underset{a>0,b>0,\mu\in\mathds{R},\sigma>0} {\arg\min}\frac{1}{40}\sum_{K,T}\frac{\big|\Pi^{\text{call}}_{0}-\check{\Pi}^{\text{call}}\big|}{\check{\Pi}^{\text{call}}},
\end{equation*}
where the sum is understood to act over all available strike prices and maturities. For comparison purposes we consider two models, one using $H$ and the other using a similar Gaussian mixture of the tempered stable subordinator $X$ without the averaging effect.

In Table \ref{tab:4} we report the calibration results including the optimal parameter values, the corresponding pricing error, and the CPU time measured in seconds\footnote{The optimization program is written in Mathematica$^\circledR$ ([Wolfram Research Inc., 2015] \cite{W}) and run on a personal laptop computer with an Intel(R) Core(TM) i5-7200 CPU @ 2.50GHz 2.71GHz.}. The model with the best fit is also marked a ``$\star$''.

\begin{table}[H]\scriptsize
  \centering
  \caption{Calibration results (rounded to 6 significant digits)}
  \label{tab:4}
  \begin{tabular}{c|c|c|c|c|c|c|c}
    \hline
    Subgroup & averaging & $\hat{a}$ & $\hat{b}$ & $\hat{\mu}$ & $\hat{\sigma}$ & ARPE & CPU time \\ \hline
    \multirow{2}*{$c\searrow0$ (gamma)} & yes & 59.8951 & 40.8853 & 0.133763 & 0.876932 & 0.244421 & 159.266 \\
    & no & 68.9358 & 40.2262 & 0.984772 & 0.523648 & 0.254048 & 469.531 \\ \hline
    \multirow{2}*{$c=0.25$} & yes & 9.42973 & 9.80518 & -0.827665 & 0.754444 & 0.228378 & 138.609 \\
    & no & 8.06212 & 9.19824 & $-0.429364$ & 0.562126 & 0.230562 & 258.219 \\ \hline
    \multirow{2}*{$c=0.5$ (inverse Gaussian)} & yes & 1.46874 & 0.682686 & $-0.594359$ & 0.635236 & 0.202169$\star$ & 270.172 \\
    & no & 1.04925 & 3.6108 & $-0.999401$ & 0.799872 & 0.210823 & 260.172 \\ \hline
    \multirow{2}*{$c=0.75$} & yes & 0.45265 & 3.0504 & $-0.961423$ & 0.983674 & 0.220485 & 146.234 \\
    & no & 0.222002 & 3.17002 & $-0.996819$ & 0.97734 & 0.223168 & 123.031 \\
    \hline
  \end{tabular}
\end{table}

We can see that for each choice of $c$ using $H$ with the averaging effect indeed improves the model fit. An explanation is of course the enhanced asymmetric leptokurtic feature of the $\mathrm{ATS}(a,b;c)$ distribution relative to the $\mathrm{TS}(a,b;c)$. Besides, the fit with $c>0$ is significantly better than in the extremal case $c\searrow0$, which can be interpreted as the trajectories of Bitcoin returns being more irregular than those of a variance gamma process. Needless to say, the pricing model based on $H$ can be implemented very efficiently. The best model fit according to Table \ref{tab:4} is further visualized in Figure \ref{fig:6} below. Observably, this pricing model fits quite commendably for the options with short maturities by incorporating jumps, whilst the existing discrepancy for those out-of-the-money long-maturity (\text{esp.} $T=257$ days) option prices can be attributed to the lack of volatility clustering.

\begin{figure}[H]
  \centering
  \includegraphics[width=4in]{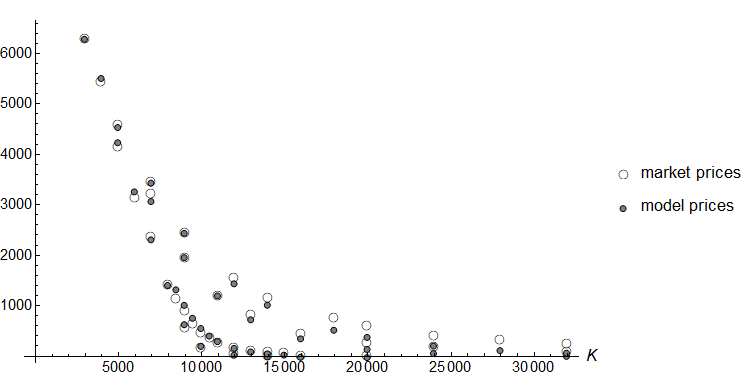}
  \caption{Bitcoin option prices (market vs model)}
  \label{fig:6}
\end{figure}

To simulate the sample paths of $\Lambda$, and therefore those of $H$, it is preferable to employ compound Poisson approximations (see, e.g., [Rydberg, 1997] \cite{R2}) based on the purely discontinuous path behavior described in Corollary \ref{cr:1}, whereas sampling directly from the distribution in Theorem \ref{th:3} seems cumbersome. The simulation procedure starts from approximating the jump intensity measure $\tilde{\nu}$ on $\mathds{R}_{++}$ by considering the intervals $\{(x_{j},x_{j+1}]\}^{M_{j}}_{j=1}$ with $x_{1}\geq\epsilon$ for some small $\epsilon>0$. After simulating $M_{j}$ stochastically independent homogeneous Poisson processes over the discretized unit time interval, denoted $\hat{N}^{(j)}\equiv\big\{\hat{N}^{(j)}_{i}\big\}^{M_{i}}_{i=0}$, with corresponding intensity $\int^{x_{j+1}}_{x_{j}}\tilde{\ell}(x)\dd x$, the paths of $\Lambda$ can be approximated as
\begin{equation*}
  \hat{\Lambda}_{i(1/M_{i})}=\tilde{\alpha}i(1/M_{i})+\sum^{M_{j}}_{j=1}\chi_{j} \Bigg(\hat{N}^{(j)}_{i(1/M_{i})}-\int^{x_{j+1}}_{x_{j}}\tilde{\ell}(x)\dd x\mathds{1}_{\{\chi_{i}<1\}}i(1/M_{i})\Bigg),
\end{equation*}
where $\tilde{\alpha}$ and $\tilde{\ell}$ are as given in Corollary \ref{cr:1} and
\begin{equation*}
  \chi_{j}:=\sqrt{\frac{\int^{x_{j+1}}_{x_{j}}x^{2}\tilde{\ell}(x)\dd x}{\int^{x_{j+1}}_{x_{j}}\tilde{\ell}(x)\dd x}},\quad j\in\mathds{N}\cap[1,M_{j}].
\end{equation*}
Each of the integrals inside the sum can be evaluated numerically with high efficiency, despite that each one has an explicit but yet lengthy expression available, which is omitted here. Once a path of $\hat{\Lambda}_{i(1/M_{i})}$ is obtained,
\begin{equation*}
  \hat{H}_{0}=0\rightsquigarrow\hat{H}_{(i+1)(1/M_{i})}=\mu(\hat{\Lambda}_{(i+1)(1/M_{i})}-\hat{\Lambda}_{i(1/M_{i})}) +\sigma\sqrt{\big|\hat{\Lambda}_{(i+1)(1/M_{i})}-\hat{\Lambda}_{i(1/M_{i})}\big|}\varsigma_{i},
\end{equation*}
for $i\in\mathds{N}\cap[1,M_{i}]$, where $\varsigma_{i}$'s are \text{i.i.d.} standard normal random variables, gives an approximation of the corresponding path of $H$.

For example, using the calibrated parameter values of the best-fit model in Table \ref{tab:4}, we notice that $\ell(7)<10^{-4}$. Therefore, taking $M_{j}=100$, $x_{1}=\epsilon=10^{-10}$ and $x_{M_{j}+1}=7+\epsilon\approx7$ with equally spaced intervals, and $M_{i}=2000$ as before, in Figure \ref{fig:7} we plot a realized sample path of $\Lambda$ and its Gaussian mixture $H$, as well as the corresponding realized path of the Bitcoin price $S_{0}e^{H}$ with $S_{0}=\$9232.98$, all of which are purely discontinuous.

\begin{figure}[H]
  \centering
  \begin{minipage}{0.49\linewidth}
  \centering
  \includegraphics[width=2.7in]{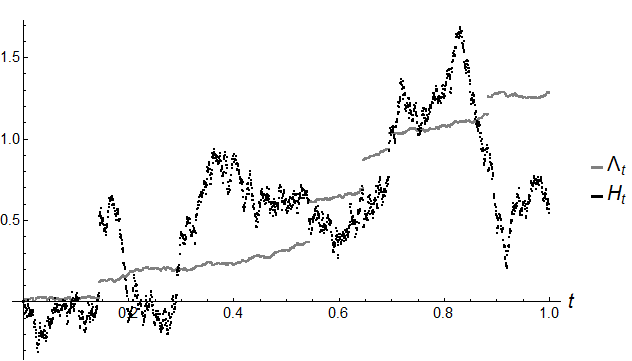}
  \end{minipage}
  \begin{minipage}{0.49\linewidth}
  \centering
  \includegraphics[width=2.7in]{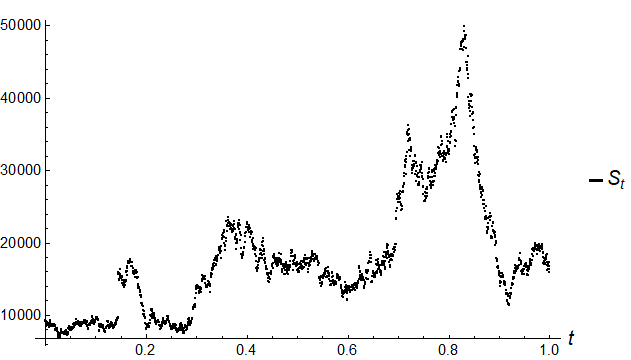}
  \end{minipage}
  \caption{Realized sample paths of $\Lambda$, $H$, and Bitcoin price}
  \label{fig:7}
\end{figure}

\vspace{0.2in}

\section{Concluding remarks}\label{sec:5}

Consideration of the running average of a tempered stable subordinator gives rise to a new family of infinitely divisible three-parameter probability distributions which are referred to as the average-tempered stable ($\mathrm{ATS}(a,b;c)$) distributions. These newfound distributions are very similar to the well-studied tempered stable distributions in that they are all positively skewed, heavy-tailed, and unimodal, whereas they exhibit severer asymmetric and leptokurtic feature subject to the same mean and variance - a result of the averaging effect. Also, the average-tempered stable distributions have a Laplace transform in simple closed form while the density and distribution functions are expressible in terms of proper definitely integrals which are very comfortable to work with. Special cases include the average-gamma distribution and the average-inverse Gaussian distribution, which are obtained, respectively, by taking the family parameter $c$ to tend to 0 or equal 1/2.

While the running average process is by construction continuous and increasing, the infinite divisibility feature also gives birth to previously unknown subordinators, the so-called ``average-tempered stable subordinators'', whose drift component and jump intensity measures are determined fully explicitly. As for applications, using the running average process is capable of modeling degradation phenomena with memory in a continuous fashion, as well as capturing the mean reversion and decreasing variance properties typically observed in reality. This beyond doubt forms a huge advantage over commonly used gamma models that are obviously purely discontinuous and Markovian. On the other hand, Gaussian mixtures of the average-tempered stable subordinators can be applied to establishing financial derivatives pricing models that can capture jumps in the returns and are suitable for highly heavy-tailed returns, such as those in the crytocurrency market. The resulting pricing methods are fairly efficient thanks to Fourier transform techniques. Of course, like any other existing L\'{e}vy-type pricing models they can also be associated with an additional stochastic volatility process to capture volatility cluster effect, which is deemed essential for long-maturity derivative prices. Last but not least, simulation of the running average process as well as the average-tempered stable subordinators and their Gaussian mixtures can be conveniently realized by means of Euler discretization and compound Poisson approximations.

\end{document}